\documentclass[american]{scrartcl}
\KOMAoptions{DIV=11, paper=letter, abstract=on}

\usepackage[utf8]{inputenc}
\usepackage[T1]{fontenc}
\usepackage{lmodern} 
\usepackage{babel, color, csquotes, microtype, fixltx2e}
\usepackage{amsmath, amssymb, amsthm, enumerate, subcaption}

\usepackage{siunitx}
\usepackage{graphicx, pgfplots}
\usetikzlibrary{calc,matrix}
\makeatletter\g@addto@macro\@floatboxreset\centering\makeatother

\pgfplotsset{compat=newest,
   plot coordinates/math parser=false,
   scale only axis,
   xmajorgrids, xminorticks=false,
   ymajorgrids, yminorticks=false,
   every axis/.append style={font=\small, line width=0.75pt, mark = +,%
      scaled x ticks = false},
   legend style={at={(0.02,1.00)},anchor=north west,align=left,
      fill=none,draw=none,row sep=-0.4em},
}

\usepackage[style=numeric, backend=bibtex8]{biblatex}
\providecommand{\acite}[1]{\citeauthor{#1}~\cite{#1}}
\DefineBibliographyStrings{german}{references = {References}}
\AtEveryBibitem{\ifentrytype{software}{}{\clearfield{url}}}
\newbibmacro{string+doi}[1]{%
  \iffieldundef{doi}{#1}{\href{http://dx.doi.org/\thefield{doi}}{#1}}}
\DeclareFieldFormat{title}{\usebibmacro{string+doi}{\mkbibemph{#1}}}
\DeclareFieldFormat[article]{title}{\usebibmacro{string+doi}{\mkbibquote{#1}}}
\renewbibmacro{in:}{}
\DeclareFieldFormat[article]{title}{#1}
\DeclareFieldFormat[inproceedings]{title}{#1}

\usepackage{hyperref}
\hypersetup{%
   pdfauthor  ={Weizhang Huang, Lennard Kamenski},
   pdftitle   ={A~geometric discretization and~a~simple implementation
      for~variational mesh generation and adaptation},
   pdfsubject ={}
}
\usepackage[capitalize]{cleveref}
\crefformat{section}{#2Sect.~#1#3}
\Crefformat{section}{#2Section~#1#3}
\crefmultiformat{section}{Sects.~#2#1#3}%
   { and~#2#1#3}{, #2#1#3}{ and~#2#1#3}
\crefformat{equation}{(#2#1#3)}
\Crefformat{equation}{Equation~(#2#1#3)}
\crefmultiformat{equation}{(#2#1#3)}%
   { and~(#2#1#3)}{, (#2#1#3)}{ and~(#2#1#3)}
\crefrangeformat{equation}{(#3#1#4) to~(#5#2#6)}

\providecommand{\abs}[1]{\lvert#1\rvert}
\providecommand{\Abs}[1]{\left\lvert#1\right\rvert}
\providecommand{\V}[1]{\boldsymbol{#1}}
\providecommand{\dx}{\,d\V{x}}

\providecommand{\p}[2]{\frac{\partial{}#1}{\partial{}#2}}
\providecommand{\Th}{\mathcal{T}_h}
\providecommand{\Thc}{\mathcal{T}_{c,h}}
\providecommand{\ThcO}{\mathcal{T}_{c0, h}}
\providecommand{\M}{\mathbb{M}}

\providecommand{\J}{\mathbb{J}}
\providecommand{\R}{\mathbb{R}}
\providecommand{\JMJ}{\J \M^{-1}\J^T}
\DeclareMathOperator{\diag}{diag}
\DeclareMathOperator{\sgn}{sgn}
\DeclareMathOperator{\tr}{tr}
\newtheorem{lemma}{\hspace{6mm}Lemma}[section]
\newtheorem{corollary}{\hspace{6mm}Corollary}[section]
\theoremstyle{remark}
\newtheorem{example}{\hspace{6mm}Example}[section]

\newenvironment{keywords}%
   {\begin{trivlist}\item[]{\bfseries\sffamily Key words:}~}
   {\end{trivlist}}
\newenvironment{AMS}%
   {\begin{trivlist}\item[]{\bfseries\sffamily AMS subject classifications:}~}
   {\end{trivlist}}

\begin{document}

\title{A~geometric discretization and~a~simple implementation
   for~variational mesh generation and~adaptation%
   \thanks{%
      Supported in~part by~%
      the~NSF (U.S.A.) under Grant DMS-1115118.%
      }%
}

\author{%
   Weizhang Huang%
   \thanks{%
      The University of~Kansas, Department of~Mathematics, Lawrence, KS~66045, U.S.A.
      (\href{mailto:whuang@ku.edu}{\nolinkurl{whuang@ku.edu}}).%
   }
   \and
   Lennard Kamenski%
   \thanks{%
      Weierstrass Institute for~Applied Analysis and~Stochastics, Berlin, Germany
      (\href{mailto:kamenski@wias-berlin.de}{\nolinkurl{kamenski@wias-berlin.de}}).%
   }
  }

\maketitle

\begin{abstract}
We present a simple direct discretization for functionals used in the variational mesh generation and adaptation.
Meshing functionals are discretized on simplicial meshes and the Jacobian matrix of the continuous coordinate transformation is approximated by the Jacobian matrices of affine mappings between elements.
The advantage of this direct geometric discretization is that it preserves the basic geometric structure of the continuous functional, which is useful in preventing strong decoupling or loss of integral constraints satisfied by the functional.
Moreover, the discretized functional is a function of the coordinates of mesh vertices and its derivatives have a simple analytical form, which allows a simple implementation of variational mesh generation and adaptation on computer.
Since the variational mesh adaptation is the base for a number of adaptive moving mesh and mesh smoothing methods, the result in this work can be used to develop simple implementations of those methods.
Numerical examples are given.
\end{abstract}




\section{Introduction}
\label{sect:introduction}

The basic idea of the variational approach of mesh generation and adaptation is to generate an adaptive mesh as an image of a given reference mesh under a coordinate transformation determined by a functional (which will hereafter be referred to as a meshing functional).
Typically, the meshing functional measures difficulties of the numerical approximation of the physical solution and involves a user-prescribed metric tensor or a monitor function to control the mesh adaptation.
The advantage of the variational approach is the relative ease of incorporating mesh requirements such as smoothness, adaptivity, or alignment in the formulation of the functional~\cite{BS82}.
The variational approach is commonly used to generate structured meshes but it can be employed to generate unstructured meshes as well~\cite{CHR99b}.
Moreover, it is the base for a number of adaptive moving mesh methods~\cite{HR11,HR99,HRR94a,LTZ01}.

A number of variational methods have been developed in the past; e.g.,\ see \acite{TWM85}, \acite{KS94}, \acite{Lis99}, \acite{HR11} and references therein.
Noticeably, \acite{Win81} proposed an equipotential method based on variable diffusion.
\acite{BS82} developed a method by combining mesh concentration, smoothness, and orthogonality.
\acite{Dvi91} used the energy of harmonic mappings as his meshing functional.
\acite{Knu96} and \acite{KR00} developed functionals based on the idea of conditioning the Jacobian matrix of the coordinate transformation. 
\acite{Hua01b} and \acite{HR11} developed functionals based on the so-called equidistribution and alignment conditions.

A common solution strategy for the existing variational methods is to first derive the Euler-Lagrange equation of the underlying meshing functional and then discretize it on either a physical or a computational mesh (cf.~\cref{fig:solution:strategies}).
If the descretization is done on a computational mesh, the Euler-Lagrange equation needs to be transformed by changing the roles of dependent and independent variables.
Although this strategy works well for the most cases, the corresponding formulation can become complicated and its implementation requires a serious effort, especially in three dimensions; cf.~\cite[Chapt.~6]{HR11}.
Moreover, the geometric structure of the meshing functional can be lost in the process of spatial discretization of the Euler-Lagrange equation.

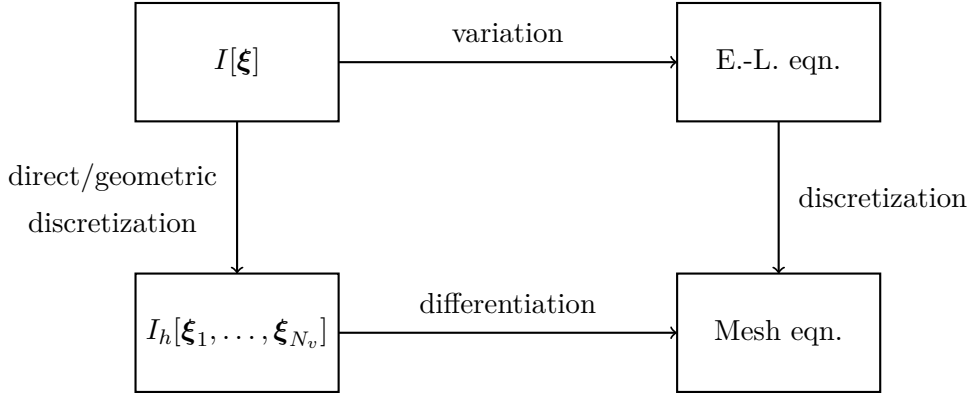
\begin{figure}
   \tikzset{my node/.code=\ifpgfmatrix\else\tikzset{matrix of nodes}\fi}
   \begin{tikzpicture}[every node/.style={my node},scale=0.45]
      \draw[thick] (0,0) rectangle (6,3.5);
      \node (node1) at (3,1.75) {$I[\V{\xi}]$\\};
      \draw[thick] (16,0) rectangle (22,3.5);
      \node (node2) at (19,1.75) {E.-L.\ eqn.\\};
      \draw[->,thick] (6,1.75)--(16,1.75);
      \node[above] at (11,1.75) {variation\\};
      \draw[thick] (16,-8) rectangle (22,-4.5);
      \node (node4) at (19,-6.25) {Mesh eqn.\\};
      \draw[->,thick] (19,0)--(19,-4.5);
      \node[right] at (19,-2.25) {discretization\\};
         \draw[thick] (0,-8) rectangle (6,-4.5);
         \node (node4) at (3,-6.25) {$I_h[\V{\xi}_1, \dotsc, \V{\xi}_{N_v} ]$\\};
         \draw[->,thick] (3,0)--(3,-4.5);
         \node[left] at (3,-2.25) {direct/geometric\\discretization\\};
         \draw[->,thick] (6,-6.25)--(16,-6.25);
         \node[above] at (11,-6.25) {differentiation\\};
   \end{tikzpicture}
   \caption{%
      Possible solution strategies for variational mesh generation and adaptation%
   }\label{fig:solution:strategies}
\end{figure}

The objective of this paper is to study a new discretization and solution strategy.
We consider simplicial meshes and approximate the underlying meshing functional directly.
Although the direct discretization of the variational problems is not new on itself, its employment in the context of variational mesh generation and adaptation is new.
The Jacobian matrix of the coordinate transformation involved in the meshing functional is not discretized directly; instead, it is approximated by the Jacobian matrices of affine mappings between simplicial elements.
The advantage of this geometric discretization is that it preserves the basic geometric structure of the continuous functional, which is useful in preventing strong decoupling or loss of integral constraints satisfied by the underling functional (cf.~\acite{Cas86}).
In particular, it preserves the coercivity and convexity for two examples of meshing functionals we consider (see \cref{sect:preservation}).
Moreover, as we will see in \cref{sect:formula}, the discretized functional is a function of the coordinates of vertices of the computational mesh and its derivatives have a simple analytical form.
This allows a simple (and parallel) implementation of the corresponding variational meshing method.

The outline of the paper is as follows.
\Cref{sect:mmpde} briefly describes the variational approach in mesh generation.
\Cref{sect:formula} presents the direct discretization for meshing functionals and gives the analytical formula for the derivatives of the discretized functional with respect to the computational coordinates of mesh vertices.
Several numerical examples are presented in \cref{sect:numerics}, followed by conclusions and further remarks in \cref{sect:conclusion}.
For completeness and for the convenience of users who prefer the physical coordinates as unknown variables, the derivatives of the discretized functional with respect to the physical coordinates are given in \cref{sect:x:der}.

\section{The variational approach for~mesh generation}
\label{sect:mmpde}

Let $\Omega$ and $\Omega_c$ be the physical and computational domains in $\R^d$ $(d \ge 1)$, which are assumed to be bounded, simply connected, and polygonal/polyhedral.
Generally speaking, $\Omega_c$ can be chosen to be the same as $\Omega$ but there are benefits to choose it to be convex, including that the to-be-determined coordinate transformation is less likely to be singular, see \acite{Dvi91}. 
We also assume that we are given a symmetric and uniformly positive definite metric tensor $\M = \M(\V{x})$ in $\Omega$,
which provides the information about the size and shape of mesh elements.
Typically, $\M$ is defined in a mesh adaptation process based on the physical solution, solution error, or other physical considerations.

Denote the coordinates on $\Omega$ and $\Omega_c$ by $\V{x}$ and $\V{\xi}$, the corresponding coordinate transformation by $\V{x} = \V{x}(\V{\xi})\colon \Omega_c \to \Omega$ and its inverse by $\V{\xi} = \V{\xi}(\V{x})\colon \Omega \to \Omega_c$.
Meshing functionals are commonly formulated in terms of the inverse coordinate transformation because the coordinate transformation determined in this way is less likely to be singular~\cite{Dvi91}.

We consider a general meshing functional
\begin{equation}
   I[\V{\xi}] = \int_\Omega G \left (\J, \det(\J), \M, \V{x} \right) \dx,
   \label{eq:fun}
\end{equation}
where $\J = \p{\V{\xi}}{\V{x}}$ is the Jacobian matrix of $\V{\xi} = \V{\xi}(\V{x})$ and $G$ is a given smooth function (with respect to all of its arguments).
This form is very general and includes many existing meshing functionals, e.g.,\ see \acite{KS94}, \acite{Lis99}, and \acite{HR11}.
To be instructive, we consider two examples in the following.
(For a detailed numerical comparison of various functionals see~\cite{HuaKamRus15}.)

\begin{example}[generalized Winslow's functional]
\label{ex:Winslow}
The first example is a generalization of Winslow's variable diffusion functional~\cite{Win81},
\begin{equation}
   I[\V{\xi}] = \int_\Omega \tr(\JMJ) \dx,
   \label{eq:fun:winslow}
\end{equation}
where $\tr(\cdot)$ denotes the trace of a matrix.
This functional has been used by many researchers, e.g.,\ see \citeauthor{HR97b}~\cite{HR97b, HR99}, \acite{LTZ01}, and \acite{BMR01}.
It is coercive and convex and therefore has a unique minimizer~\cite[Example 6.2.1]{HR11}.
\qed{}
\end{example}
\vspace{1ex}

\begin{example}[Huang's functional]
\label{ex:Huang}
The second functional is
\begin{equation}
   I[\V{\xi}]  
   = \theta \int_\Omega \sqrt{\det(\M)} {\left(\tr(\JMJ) \right)}^{\frac{dp}{2}} \dx
   + (1 - 2 \theta) d^{\frac{dp}{2}} \int_\Omega \sqrt{\det(\M)}
            {\left( \frac{\det(\J)}{\sqrt{\det(\M)}}\right)}^{p} \dx,
   \label{eq:fun:huang}
\end{equation}
where $0 \le \theta \le 1$ and $p > 0$ are dimensionless parameters.
This functional was proposed by Huang~\cite{Hua01b} based on the so-called alignment (first term) and equidistribution (second term) conditions.
For $0 < \theta \le \frac{1}{2}$, $d p \ge 2$, and $p \ge 1$, the functional is coercive and polyconvex and has a minimizer~\cite[Example 6.2.2]{HR11}.
Moreover, for $\theta = \frac{1}{2}$ and $d p = 2$ it reduces to
\[
   I[\V{\xi}]  = \theta \int_\Omega \sqrt{\det(\M)} \tr(\JMJ) \dx,
\]
which is the energy functional for a harmonic mapping from $\Omega$ to $\Omega_c$ (cf.~\acite{Dvi91}).
\qed{}
\end{example}
\vspace{1ex}

As mentioned in the introduction, a common strategy to compute the coordinate transformation for a given functional is to discretize and solve its Euler-Lagrange equation.
The derivation of the Euler-Lagrange equation for \cref{eq:fun} is standard (e.g., see~\cite[Chapter 6]{HR11}).
It reads as
\begin{equation}
   \frac{\delta I}{\delta \V{\xi}} 
      \equiv - \nabla \cdot \left( \p{G}{\J} + \p{G}{r} \det(\J) \J^{-1} \right)
      = 0
   ,
   \label{eq:EL:1}
\end{equation}
where $\frac{\delta I}{\delta \V{\xi}}$ denotes the functional derivative of $I[\V{\xi}]$ and $\p{G}{\J}$ and
$\p{G}{r} = \p{G}{\det(\J)}$ are the derivatives of $G$ with respect to its first and second arguments.\footnote{See \cref{SEC:scalar-by-matrix-der} for the notation of scalar-by-matrix derivatives.}
Notice that this equation is defined for the inverse coordinate transformation $\V{\xi}$.
Since the computational mesh of $\Omega_c$ is typically given while the physical mesh of $\Omega$ is to be determined, it is common to transform the mesh equation \cref{eq:EL:1} by exchanging the roles of the independent variable $\V{x}$ and the dependent variable $\V{\xi}$ and then discretize it on the computational mesh using finite difference or finite element methods.
The resulting (nonlinear) algebraic system is solved for the physical mesh.

Although \cref{eq:EL:1} is simple, the transformed mesh equation (after exchanging the roles of~$\V{x}$ and~$\V{\xi}$) becomes complicated and a serious effort is needed to implement its finite difference or finite element discretization on computer.
The process can be simplified by some degree by first discretizing \cref{eq:EL:1} directly on the physical mesh (which is normally not uniform and not rectangular), solving the resulting equations for the computational mesh, and then obtaining the new physical mesh by interpolation (see the next section for the detailed discussion for a similar procedure). 
However, the implementation of a finite difference or finite element discretization of \cref{eq:EL:1} still requires a non-trivial effort.
This is the motivation for us to seek a simpler discretization and a simpler implementation (see the next section).

Generally speaking, \cref{eq:EL:1} is highly nonlinear.
A useful strategy to solve such a nonlinear system is to use a time-varying approach or a moving mesh PDE (MMPDE) approach~\cite{HRR94b,HRR94a} in the context of dynamical mesh adaptation. 
An MMPDE is defined as a gradient flow equation of the functional \cref{eq:fun}, i.e.,
\begin{equation}
   \p{\V{\xi}}{t} 
      = - \frac{P}{\tau} \frac{\delta I}{\delta \V{\xi}}
   ,
   \label{eq:mmpde:1}
\end{equation}
where $t$ is the quasi-time, $\tau > 0$ is a constant parameter used to adjust the scale of mesh movement, and $P$ is a positive balancing function.
When it is desired to compute the physical coordinates directly, we can combine \cref{eq:mmpde:1} with the identity
\[
   \p{\V{x}}{t} + \J^{-1} \p{\V{\xi}}{t} = 0 
\]
to obtain
\begin{equation*}
   \p{\V{x}}{t} = \frac{P}{\tau} \J^{-1} \p{I}{\V{\xi}}
   .
\end{equation*}

\section{Direct discretization and~analytical formulas for~derivatives of~the~discretized functional}
\label{sect:formula}

\subsection{Direct discretization}

Denote the physical and computational meshes by $\Th$ and $\Thc$ and assume they have the same numbers of elements and vertices and the same connectivity.
Since we consider only simplicial meshes, for any element $K\in\Th$ there exist a corresponding element $K_c\in\Thc$ and an invertible affine mapping $F_K\colon K_c \to K$ such that $K = F_K(K_c)$.
With this notation, we approximate the functional \cref{eq:fun} directly as
\begin{align}
   I[\V{\xi}] 
     &= \int_\Omega G(\J, \det(\J), \M, \V{x}) \dx
      = \sum\limits_{K \in \Th} \int_K G \left(\J, \det(\J), \M, \V{x}\right) \dx
   \notag\\ 
   &\approx \sum\limits_{K \in \Th} 
      \Abs{K} G \left( {(F_K')}^{-1}, \det{(F_K')}^{-1}, \M(\V{x}_K), \V{x}_K \right) ,
   \label{eq:dis:fun:1}
\end{align}
where $\Abs{K}$ and $\V{x}_K$ are the volume and the center of $K$ and $F_K'$ is the Jacobian matrix of $F_K$.

Note that in \cref{eq:dis:fun:1} we have used the mid-point quadrature formula for the involved integrals and approximated $\J$ using ${(F_K')}^{-1}$ (instead of directly discretizing $\J$ on the mesh).
The latter enables the discretized functional to preserve the basic geometric structure of the underlying functional.
For example, the integrand of functional \cref{eq:fun:winslow} is the trace of the matrix $\JMJ$.
For this functional, the summand of \cref{eq:dis:fun:1} is the product of the volume of $K$ and the trace of the matrix ${(F_K')}^{-1} \M^{-1}(\V{x}_K)  {(F_K')}^{-T}$, an approximation of $\JMJ$ on $K$.
For functional \cref{eq:fun:huang}, the summand of \cref{eq:dis:fun:1} is a linear combination of the terms corresponding to the alignment and equidistribution conditions as in the continuous functional.
The preservation of the geometric properties is useful in preventing strong decoupling or loss of integral constraints satisfied by the underlying functional (cf.\ Castillo~\cite{Cas86}).
Moreover, as we will see in \cref{eq:FK:1}, $F_K'$ can be computed using the edge matrices of $K$ and $K_c$ (cf. \cref{eq:edge:1}).
Therefore, no direct discretization of derivatives is involved in \cref{eq:dis:fun:1}.

Notice that each summand in \cref{eq:dis:fun:1} is a function of the coordinates of vertices of $K_c$, i.e., 
\begin{equation}
   I_K(\V{\xi}_0^K, \dotsc,  \V{\xi}_d^K) 
      = G \left( {(F_K')}^{-1}, \det{(F_K')}^{-1}, \M(\V{x}_K), \V{x}_K \right)
   ,
   \label{eq:IK:1}
\end{equation}
while the sum is a function of the coordinates of all vertices of the mesh $\Thc$,
\begin{equation}
   I_h(\V{\xi}_1, \dotsc, \V{\xi}_{N_v}) 
      = \sum\limits_{K \in \Th} \Abs{K} I_K \left(\V{\xi}_0^K, \dotsc,  \V{\xi}_d^K \right)
      .
   \label{eq:Ih:1}
\end{equation}
The key to our approach is to find the derivatives of $I_h$ with respect to $\V{\xi}_1, \dotsc, \V{\xi}_{N_v}$.
This is done by obtaining and assembling the elementwise derivatives of $I_K$ with respect to $\V{\xi}_0^K, \dotsc,  \V{\xi}_d^K$.
We first recall some notation and results for scalar-by-matrix derivatives.

\subsection{Scalar-by-matrix derivatives}
\label{SEC:scalar-by-matrix-der}

Let $f = f(A)$ be a scalar function of a matrix $A \subset \R^{m\times n}$.
The scalar-by-matrix derivative of $f$ with respect to $A$ is defined as
\begin{equation}
   \p{f}{A} 
      = \begin{bmatrix} 
         \p{f}{A_{11}}  & \cdots & \p{f}{A_{m1}} \\
         \vdots         &        & \vdots \\
         \p{f}{A_{1 n}} & \cdots & \p{f}{A_{m n}}
      \end{bmatrix}_{n\times m}
   \quad \text{or} \quad 
   {\left(\p{f}{A} \right)}_{i,j} = \p{f}{A_{j,i}}
   .
\label{def-1}
\end{equation}
The chain rule of differentiation involving matrices (with respect to a real parameter $t$) is
\begin{equation}
   \p{f}{t} 
      = \sum_{ij} \p{f}{A_{j,i}} \p{A_{j,i}}{t}
      = \sum_{ij} {\left( \p{f}{A} \right)}_{i,j} \p{A_{j,i}}{t}
      = \tr \left( \p{f}{A} \p{A}{t} \right)
      .
\label{chain-rule-1}
\end{equation}

Hereafter, we always assume that (i) arguments of the matrix trace ($\tr(\cdot)$) and
determinant ($\det(\cdot)$) are square matrices; (ii) matrix products are meaningful;
(iii) matrices are invertible when their inverses are involved.

The following four lemmas serve as the basic tools in the application of scalar-by-matrix differentiation.
The first two can be verified directly while the third and fourth
can be proven using the determinant expansion by minors and by differentiating the identity $A A^{-1} = I$,
respectively.

\begin{lemma}
\label{lem:3.1}
\begin{align*}
 \tr (A^T) & = \tr(A),
\\
\tr (AB) & = \tr (BA),
\\
\tr (ABC) & = \tr (CAB) = \tr (BCA) .
\end{align*}
\end{lemma}

\begin{lemma}
\label{lem:3.2}
\[
\frac{\partial \tr (A)}{\partial A} = I.
\]
\end{lemma}

\begin{lemma}
\label{lem:3.3}
\[
\frac{\partial \det (A)}{\partial A} = \det(A) \; A^{-1}.
\]
\end{lemma}

\begin{lemma}
\label{lem:3.4}
\[
\frac{\partial A^{-1}}{\partial t} = - A^{-1} \frac{\partial A}{\partial t} A^{-1}.
\]
\end{lemma}

We now derive several identities which we need in our application.
\begin{corollary}
\label{cor-1}
If a symmetric matrix $\M$ is independent of $A$, then
\begin{align}
   \p{\tr(A \M A^T)}{A} 
      &= 2 \M A^T ,
   \label{eq:diff:5}
   \\
   \p{\tr(A^{-T} \M^{-1} A^{-1})}{A} 
      &= - 2 A^{-1} A^{-T} \M^{-1} A^{-1}.
   \label{eq:diff:7}
\end{align}
Moreover, if $A$ is independent of $\M$, then
\begin{align}
   \p{\tr (A \M A^T)}{\M} & = A^T A,
\label{eq:diff:7+1}
\\
   \p{\tr (A \M^{-1} A^T)}{\M} & = - \M^{-1} A^T A \M^{-1}.
\label{eq:diff:7+2}
\end{align}
\end{corollary}

\begin{proof}
Let $t$ be an entry of $A$.
Using \cref{lem:3.1,lem:3.2}, we have
\begin{align*}
\p{\tr(A \M A^T)}{t} & = \tr \left ( \p{\tr(A \M A^T)}{(A \M A^T)} \p{(A \M A^T)}{t}\right)
 = \tr \left ( \p{(A \M A^T)}{t}\right)
\\
& = \tr \left ( \p{A}{t}\M A^T + A \M \p{A^T}{t} \right)
 = \tr \left ( \p{A}{t}\M A^T\right ) + \tr \left (A \M \p{A^T}{t} \right)
\\
& = \tr \left ( \M A^T \p{A}{t}\right ) + \tr \left ( \p{A}{t} \M A^T \right)
 = \tr \left (2 \M A^T \p{A}{t}\right ) .
\end{align*}
From the chain rule \cref{chain-rule-1}, this gives \cref{eq:diff:5}.

Moreover, using \cref{lem:3.1,lem:3.4} and identity \cref{eq:diff:5}, we have
\begin{align*}
\p{\tr(A^{-T} \M^{-1} A^{-1})}{t}  & = \tr \left ( \p{\tr(A^{-T} \M^{-1} A^{-1})}{A^{-T}}\p{A^{-T}}{t} \right )
\\
& = \tr \left ( 2 \M^{-1} A^{-1} \p{A^{-T}}{t} \right )
 = \tr \left ( 2 \M^{-1} A^{-1} (-A^{-1} \p{A^T}{t} A^{-T}) \right )
\\
& = \tr \left ( - 2 A^{-1} \p{A}{t} A^{-1} A^{-T} \M^{-1} \right )
 = \tr \left ( - 2 A^{-1} A^{-T} \M^{-1} A^{-1} \p{A}{t}  \right ),
\end{align*}
which gives \cref{eq:diff:7}.

Identities \cref{eq:diff:7+1,eq:diff:7+2} can be proven similarly.
\end{proof}

Using the above results, we can find the expressions for $\p{G}{\J}$ and $\p{G}{r}$, which are needed to compute the derivatives of the discrete functional \cref{eq:dis:fun:1}, for functionals \cref{eq:fun:winslow} and \cref{eq:fun:huang}.

\begin{example}[generalized Winslow's functional]
\label{ex:winslow:2}
   For the functional \cref{eq:fun:winslow} we have
\begin{equation}
   \begin{cases}
      \p{G}{\J} = 2 \M^{-1} \J^T,
      \\
      \p{G}{r} = 0 .
   \end{cases}
   \label{eq:winslow:2}
\end{equation}
\end{example}

\begin{example}[Huang's functional]
\label{ex:huang:2}
For the functional \cref{eq:fun:huang} we have
\begin{equation}
   \begin{cases}
      \p{G}{\J} =  d p \theta \sqrt{\det(\M)}  {\left( \tr(\JMJ )\right )}^{\frac{d p}{2}-1} \M^{-1} \J^T,
      \\
      \p{G}{r} = p (1-2\theta) d^{\frac{d p}{2}} \det{(\M)}^{\frac{1-p}{2}} \det{(\J)}^{p-1} .
   \end{cases}
   \label{eq:huang:2}
\end{equation}
\end{example}

\subsection{Analytical formulas for~derivatives of~the~discretized functional}

The coordinates $\V{x}_0^K,\dotsc,\V{x}_d^K$ of the vertices of $K$ and the coordinates $\V{\xi}_0^K,\dotsc,\V{\xi}_d^K$ of the vertices of $K_c$ are related by
\[
   \V{x}_i^K - \V{x}_0^K = F_K' (\V{\xi}_i^K -  \V{\xi}_0^K),
   \quad i = 1, \dotsc, d
   ,
\]
or, in matrix form,
\[
   [
   \V{x}_1^K - \V{x}_0^K, \dotsc, \V{x}_d^K - \V{x}_0^K] 
      = F_K' [\V{\xi}_1^K -  \V{\xi}_0^K, \dotsc, \V{\xi}_d^K - \V{\xi}_0^K
   ]
   .
\]
Thus,
\begin{equation}
   F_K' = E_K \hat{E}_K^{-1}
   \qquad \text{and} \qquad 
   {(F_K')}^{-1} = \hat{E}_K E_K^{-1}
   ,
   \label{eq:FK:1}
\end{equation}
where $E_K$ and $\hat{E}_K$ are the edge matrices for $K$ and $K_c$,
\begin{equation}
   E_K = [\V{x}_1^K - \V{x}_0^K, \dotsc, \V{x}_d^K - \V{x}_0^K]
   \qquad \text{and} \qquad 
   \hat{E}_K = [\V{\xi}_1^K -  \V{\xi}_0^K, \dotsc, \V{\xi}_d^K -  \V{\xi}_0^K]
   .
   \label{eq:edge:1}
\end{equation}

We now derive the derivatives of $I_K$.
Using \cref{lem:3.1,lem:3.3,eq:FK:1}, we have
\begin{align}
   \p{I_K}{t} 
      &= \tr\left( \p{G}{\J} \p{{(F_K')}^{-1}}{t} \right)
         + \p{G}{r} \p{\det{(F_K')}^{-1}}{t}
\notag  \\
      &= \tr\left( \p{G}{\J} \p{\hat{E}_K}{t} E_K^{-1} \right)
         + \p{G}{r} \det{(E_K)}^{-1} \p{\det(\hat{E}_K)}{t}
 \notag  \\
      &= \tr\left( E_K^{-1} \p{G}{\J} \p{\hat{E}_K}{t} \right)
         + \p{G}{r} \frac{\det(\hat{E}_K)}{\det(E_K)} 
            \tr\left( \hat{E}_K^{-1} \p{\hat{E}_K}{t} \right) ,
   \label{eq:I:K:1}
\end{align}
where 
\begin{align*}
\p{G}{\J} & = \p{G}{\J}(\hat{E}_K E_K^{-1}, \frac{\det(\hat{E}_K)}{\det(E_K)} , \M(\V{x}_K), \V{x}_K),
\\
\p{G}{r} & = \p{G}{r}(\hat{E}_K E_K^{-1}, \frac{\det(\hat{E}_K)}{\det(E_K)} , \M(\V{x}_K), \V{x}_K) .
\end{align*}
Letting $t$ be any entry of the matrix $[\V{\xi}_1^K, \dotsc, \V{\xi}_d^K]$, we notice that
\[
   \p{\hat{E}_K}{t} = \p{[\V{\xi}_1^K,\dotsc,\V{\xi}_d^K]}{t} .
\]
Combining this with \cref{eq:I:K:1}, we get
\begin{equation}
   \p{I_K}{[\V{\xi}_1^K,\dotsc,\V{\xi}_d^K]} 
      = E_K^{-1} \p{G}{\J} + \p{G}{r} \frac{\det(\hat{E}_K)}{\det(E_K)} \hat{E}_K^{-1} .
   \label{eq:der:1}
\end{equation}
Moreover, for $j = 1, \dotsc, d$, from \cref{eq:I:K:1} and the equality
\[
   \p{\hat{E}_K}{\xi_{0}^{K (j)}} = - \sum_{i=1}^d \p{\hat{E}_K}{\xi_{i}^{K (j)}},
\]
where $\xi_{i}^{K (j)}$ denotes the $j^{\text{th}}$ component of $\V{\xi}_i^K$, we have
\begin{align*}
   \p{I_K}{\xi_{0}^{K (j)}} 
   &= - \sum_{i=1}^d
      \left[ 
         \tr\left( \p{G}{\J} \p{\hat{E}_K}{\xi_{i}^{K (j)}} E_K^{-1} \right)
            + \p{G}{r} \frac{\det(\hat{E}_K)}{\det(E_K)} 
            \tr\left( \hat{E}_K^{-1} \p{\hat{E}_K}{\xi_{i}^{K (j)}} \right)
      \right]
   \\
   &= - \sum_{i=1}^d {\left(\p{I_K}{[\V{\xi}_1^K,\dotsc,\V{\xi}_d^K]}\right)}_{i,j} ,
\end{align*}
which gives
\begin{equation}
   \p{I_K}{\V{\xi}_0^K} 
      = - \V{e}^T \p{I_K}{[\V{\xi}_1^K,\dotsc,\V{\xi}_d^K]}
   ,
   \qquad
   \V{e} = {[1, \dotsc, 1]}^T
   . 
   \label{eq:der:2}
\end{equation}

To summarize, the derivatives of $I_h$ with respect to $\V{\xi}_{1},\dotsc,\V{\xi}_{N_v}$ are computed as follows:
\begin{enumerate}[(i)]
   \item Compute the edge matrices $E_K$, $\hat{E}_K$ and their inverses from \cref{eq:edge:1};
   \item Compute $F_K'$ and its inverse via \cref{eq:FK:1} and quantities $\p{G}{\J}$ and $\p{G}{r}$ through \cref{eq:winslow:2} or \cref{eq:huang:2};
   \item Compute the derivatives of $I_K$ with respect to $\V{\xi}_{0}^K,\dotsc,\V{\xi}_{d}^K$ through \cref{eq:der:1,eq:der:2};
   \item Finally, the derivatives of $I_h$ with respect to $\V{\xi}_{1},\dotsc,\V{\xi}_{N_v}$ are obtained by assembling the element-wise derivatives (cf.~\cref{eq:Ih:1}).
\end{enumerate}

The mesh equation for $\V{\xi}_{1},\dotsc,\V{\xi}_{N_v}$ reads as
\begin{equation}
   \p{I_h}{[\V{\xi}_{1},\dotsc,\V{\xi}_{N_v}]} = 0
   .
   \label{eq:mesheq:1}
\end{equation}
This equation is analytical and, as for the standard finite element computation, the elementwise derivatives can be computed in parallel for all elements and then assembled together to form the global derivatives.
Moreover, the Jacobian matrix for \cref{eq:mesheq:1} is sparse.
Its analytical expression is harder to obtain but its finite difference approximation can be computed in parallel as well.

Note that equation \cref{eq:mesheq:1} can be highly nonlinear.
As in the continuous situation, we can use the MMPDE approach (cf.~\cref{eq:mmpde:1}), i.e.,
\begin{equation}
   \p{[\V{\xi}_{1},\dotsc,\V{\xi}_{N_v}]}{t}
      = -\frac{1}{\tau} {\left( \p{I_h}{[\V{\xi}_{1},\dotsc,\V{\xi}_{N_v}]} \right )}^T P,
   \label{eq:mmpde:3}
\end{equation}
where the balancing factor $P$ is now a diagonal $N_v \times N_v$ matrix.
Since mesh concentration should not be affected by scaling transformations of $\M$, we choose $P$ such that both sides of \cref{eq:mmpde:3} are homogeneous in the dimension of $\M$.
For example, from \cref{eq:winslow:2,eq:huang:2,eq:der:1} we find that the dimension of $P$ is
\[
   [P] =
   \begin{cases}
      [\M],
         & \text{for functional \cref{eq:fun:winslow}} \\
      {[\M]}^{\frac{d (p-1)}{2}} ,
         & \text{for functional \cref{eq:fun:huang}}
         ,
   \end{cases}
\]
where $[P]$ and $[\M]$ denote the dimension of $P$ and $\M$, respectively.
A reasonable choice for $[\M]$ is $[\M] = {\det(\M)}^{\frac{1}{d}}$.
Hence, we can choose
\begin{equation}
   P = \diag(P_1,\dotsc,P_{N_v})
   \quad \text{with} \quad
   P_i =
      \begin{cases}
         {\det(\M(\V{x}_i))}^{\frac{1}{d}},
            & \text{for functional \cref{eq:fun:winslow}} \\
         {\det(\M(\V{x}_i))}^{\frac{p-1}{2}},
            & \text{for functional \cref{eq:fun:huang}} .
   \end{cases}
   \label{eqP:1}
\end{equation}
With this choice of $P$, the MMPDE~\cref{eq:mmpde:3} is invariant under the scaling transformation $\M \to c \, \M$ for any positive constant $c$.

The mesh equation \cref{eq:mmpde:3} can be written more explicitly using local mesh velocities.
From \cref{eq:Ih:1}, we have
\[
   \p{I_h}{\V{\xi}_i} = \sum_{K\in \omega_i} |K| \p{I_K}{\V{\xi}_i} ,
\]
where $\omega_i$ is the element patch associated with vertex $\V{\xi}_i$.
From this we can rewrite \cref{eq:mmpde:3} into
\begin{equation}
\label{eq:mmpde:4}
\p{\V{\xi}_i}{t} = \frac{P_i}{\tau} \sum_{K \in \omega_i} |K| \V{v}_{i_K}^K , \quad i = 1, \dotsc, N_v
\end{equation}
where $i_K$ and $\V{v}_{i_K}^K$ are the local index and velocity of vertex $\V{\xi}_i$ on the element $K$,
respectively. The local velocities are defined as
\[
   \begin{bmatrix} {(\V{v}_0^K)}^T \\ \vdots \\ {(\V{v}_d^K)}^T \end{bmatrix}
=  - \p{I_K}{[\V{\xi}_0^K,\dotsc,\V{\xi}_d^K]} .
\]
From \cref{eq:der:1,eq:der:2}, we have
\begin{equation}
\label{eq:mmpde:5}
\begin{bmatrix} {(\V{v}_1^K)}^T \\ \vdots \\ {(\V{v}_d^K)}^T \end{bmatrix}
= - E_K^{-1} \p{G}{\J} - \p{G}{r} \frac{\det(\hat{E}_K)}{\det(E_K)} \hat{E}_K^{-1}, 
\quad
\V{v}_0^K = - \sum_{i=1}^d \V{v}_d^K .
\end{equation}

The MMPDE~\cref{eq:mmpde:4} should be modified properly for boundary vertices.
For example, if $\V{\xi}_i$ is a fixed boundary vertex, we replace the corresponding equation by
\begin{equation}
   \p{\V{\xi}_i}{t} = 0.
   \label{eq:fixed-bc}
\end{equation}
When $\V{\xi}_i$ is allowed to move on a boundary curve (in 2D) or surface (in 3D) represented by
\[
   \phi(\V{\xi}) = 0,
\]
then the mesh velocity $\p{\V{\xi}_i}{t}$ needs to be modified such that its normal component along the curve or surface is zero, i.e.,
\[
   \nabla \phi (\V{\xi}_i) \cdot \p{\V{\xi}_i}{t} = 0.
\]

The MMPDE~\cref{eq:mmpde:4} (with suitable modifications for boundary vertices) can be integrated from $t_n$ to $t_{n+1}$ for the new computational mesh $\Thc$.
Once $\Thc$ has been computed, the new physical mesh $\widetilde{\Th}$ is obtained via linear interpolation: if the correspondence between $\Thc$ and $\Th$ (current physical mesh) is
\[
   \V{x} = \Phi_h(\V{\xi})\colon \Omega_c \to \Omega
   \qquad \text{and} \qquad
   \Th = \Phi_h(\Thc)
   ,
\]
then the new physical mesh is given by
\[
   \widetilde{\Th} = \Phi_h(\ThcO),
\]
where $\ThcO$ is a given reference mesh of $\Omega_c$.
Typically, $\ThcO$ should be chosen as uniform as possible but this is not necessary, although the non-uniformity of $\ThcO$ will affect the resulting physical mesh (see \cref{ex:smoothing,fig:smoothing}). 
Moreover, $\ThcO$ does not have to have the same number of vertices, elements, or the same connectivity as $\Th$.
Hence, a two-level strategy can be used to improve the efficiency: a coarser mesh for the mesh equation and a finer mesh ---obtained via linear interpolation--- for the physical equation (e.g,\ see \acite{Hua01}).

\subsection{Preservation of~coercivity and~convexity}
\label{sect:preservation}

In \cref{sect:mmpde} we have mentioned that both Winslow's and Huang's functionals are coercive and convex/polyconvex which guarantee the existence of minimizers.
In the following, we show that these properties are preserved by the discretization discussed in the preceding subsections.
To this end, we state a lemma which generalizes \cref{cor-1} and whose proof is straightforward.

\begin{lemma}
\label{lem:tr}
If $A$, $B$, $C$ are square matrices of the same size and $A$ and $C$ are independent of $B$, then
\begin{equation}
\p{\tr(ABC)}{B} = C A .
\label{eq:diff:10}
\end{equation}
\end{lemma}

We first consider Winslow's functional.
In this case, from \cref{eq:fun:winslow,eq:FK:1} the discretized functional can be expressed as
\begin{equation}
   I_h = \sum_{K} \abs{K} I_K(\hat{E}_K), \quad I_K(\hat{E}_K) = \tr(\hat{E}_K S \hat{E}_K^T)
   ,
   \label{eq:winslow:3}
\end{equation}
where $S = E_K^{-1} \M^{-1}(\V{x}_K) E_K^{-T}$. 
Recall that $\M$ is assumed to be symmetric and uniformly positive definite.
If we assume that the (current) physical mesh $\Th$ has no degenerate elements, then $E_K$ is nonsingular and $S$ is uniformly (over all elements) positive definite.
Hence, there exists a positive constant $\alpha$ (independent of $K$) such that
\begin{equation}
   I_K(\hat{E}_K) \ge \alpha \tr(\hat{E}_K \hat{E}_K^T) \quad \forall K \in \Th
   .
   \label{eq:winslow:4}
\end{equation}
The above inequality is a discrete analogue of the coercivity condition for continuous functionals (cf.~\cite[(6.52)]{HR11}).
Moreover, for any $K\in \Th$ and any edge matrix
$E_{\V{\eta}} = [\V{\eta}_{1}^K-\V{\eta}_{0}^K, \dotsc, \V{\eta}_{d}^K-\V{\eta}_{0}^K]$, 
\[
   \tr\left(\p{I_K}{\hat{E}_K} E_{\V{\eta}}\right) 
      = \tr\left(2 S \hat{E}_K^T E_{\V{\eta}}\right)
      =  2 \tr\left(E_{\V{\eta}}^T \hat{E}_K S \right)
   .
\]
From this,
\begin{align*}
   \tr\left( \p{\tr\left(\p{I_K}{\hat{E}_K} E_{\V{\eta}}\right)}{\hat{E}_K} E_{\V{\eta}}\right) 
   &= 2 \tr\left( \p{\tr\left(E_{\V{\eta}}^T \hat{E}_K S \right)}{\hat{E}_K}   E_{\V{\eta}}\right) 
   \\ 
   &= 2 \tr\left( S  E_{\V{\eta}}^T  E_{\V{\eta}}\right) 
   = 2 \tr\left( E_{\V{\eta}} S  E_{\V{\eta}}^T  \right)  \ge 0
   ,
\end{align*}
which is a discrete analogue of the convexity condition for continuous functionals (cf.~\cite[(6.53)]{HR11}).
Thus, $I_h$ preserves the coercivity and convexity of Winslow's functional.
It is noted that the latter property implies that $I_h$ is a convex function of $\V{\xi}_{1}, \dotsc, \V{\xi}_{N_v}$.
\emph{%
Hence, $I_h$, with or without suitable boundary conditions, has a unique minimizer.%
}

For Huang's functional \cref{eq:fun:huang}, $I_K$ takes the form
\begin{align}
   I_K & = \theta \det{(\M(\V{x}_K))}^{\frac{1}{2}} {(\tr(\hat{E}_K S \hat{E}_K^T))}^{\frac{d p}{2}}
   \notag \\
   & + (1-2\theta) d^{\frac{d p}{2}} \det{(\M(\V{x}_K))}^{\frac{1-p}{2}} \det{(E_K)}^{-p} \det{(\hat{E}_K)}^p .
   \label{eq:huang:3}
\end{align}
As for Winslow's functional, we can show that $I_K$ preserves the coercivity and polyconvexity of the continuous functional for $0 < \theta \le \frac{1}{2}$, $dp \ge 2$,  and $p \ge 1$.
Particularly, $I_K$ is polyconvex in the sense that it is convex when considered as a function of $\hat{E}_K$ and $\det(\hat{E}_K)$.
Note that $I_K$ is not convex in general when considered as a function of $\hat{E}_K$.
In the continuous situation, coercivity and polyconvexity imply the existence of minimizers of the functional (e.g.,\ see~\cite{Evans1998}).
However, it is unclear to the authors if this is true in the discrete situation.

\section{Numerical examples}
\label{sect:numerics}

In this section we present examples to demonstrate the direct discretization and solution strategy discussed in the previous section.
Unless otherwise stated, we use Huang's functional \cref{eq:fun:huang} with $\theta = \frac{1}{3}$ and $p=2$ as the meshing functional.
Moreover, the mesh equation with $\tau = 0.1$ is integrated from $t = 0$ to $t=1$ using Matlab ODE solver \emph{ode15s}.

\begin{example}[2D, mesh smoothing]
\label{ex:smoothing}
In the first example we demonstrate how our method can be used for mesh smoothing.
To this end, we choose a mesh on $(0,1) \times (0,1)$ (\cref{fig:smoothingU:regular} or \cref{fig:smoothingR:regular}) as the reference computational mesh $\ThcO$ and the initial computational mesh $\Thc$.
The initial physical mesh $\Th$ is obtained by randomly perturbing the coordinates of the interior vertices of $\Thc$ (see \cref{fig:smoothingU:perturbed,fig:smoothingR:perturbed}).
The mesh is smoothed by integrating the mesh equation with $\M = I$.
The final meshes obtained at $t=1$ are shown in \cref{fig:smoothingU:smoothed,fig:smoothingR:smoothed}.
One can see that they are very smooth and almost identical to $\ThcO$.
The latter is due to the fact that the optimal coordinate transformation for this example ($\M=I$ and $\Omega = \Omega_c$) is $\V{x}(\V{\xi}) = \V{\xi}$ and therefore the final mesh is identical to the reference computational mesh $\ThcO$.

\begin{figure}[p]
   \begin{subfigure}[t]{0.31\linewidth}
      \includegraphics[clip, width=1.0\linewidth]{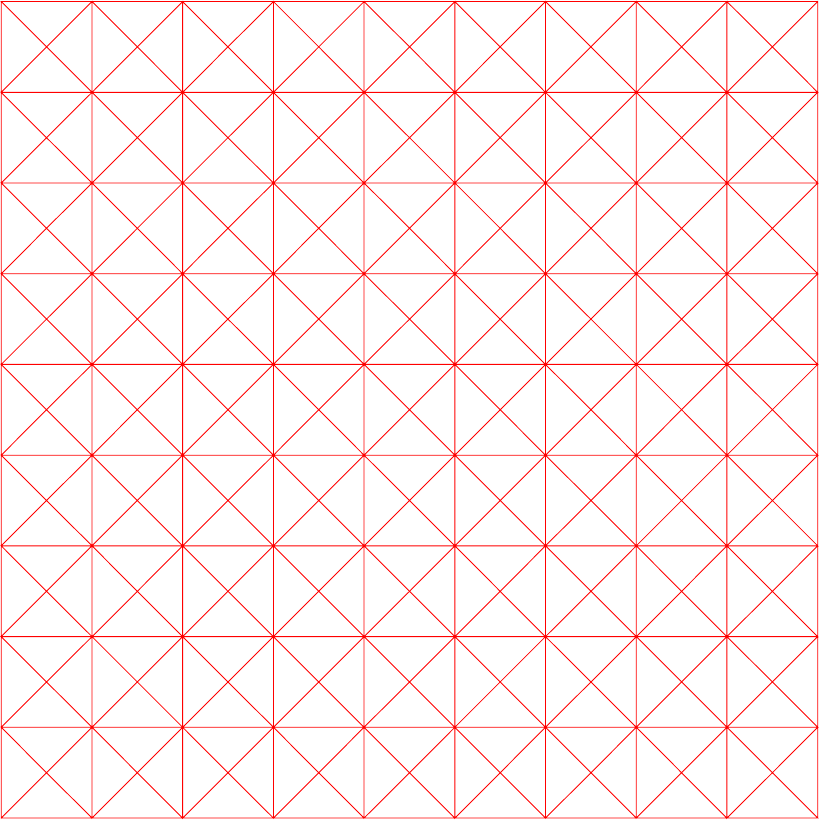}%
      \caption{uniform $\ThcO$}\label{fig:smoothingU:regular}
   \end{subfigure}%
   \hfill%
   \begin{subfigure}[t]{0.31\linewidth}
      \includegraphics[clip, width=1.0\linewidth]{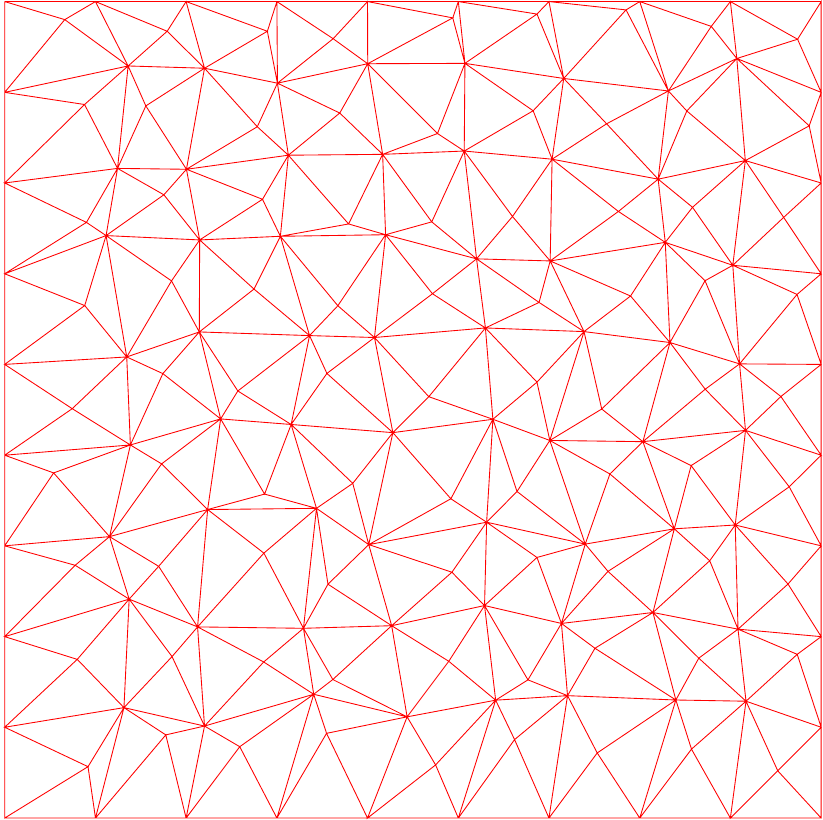}%
      \caption{initial, perturbed physical mesh}\label{fig:smoothingU:perturbed}
   \end{subfigure}
   \hfill%
   \begin{subfigure}[t]{0.31\linewidth}
      \includegraphics[clip, width=1.0\linewidth]{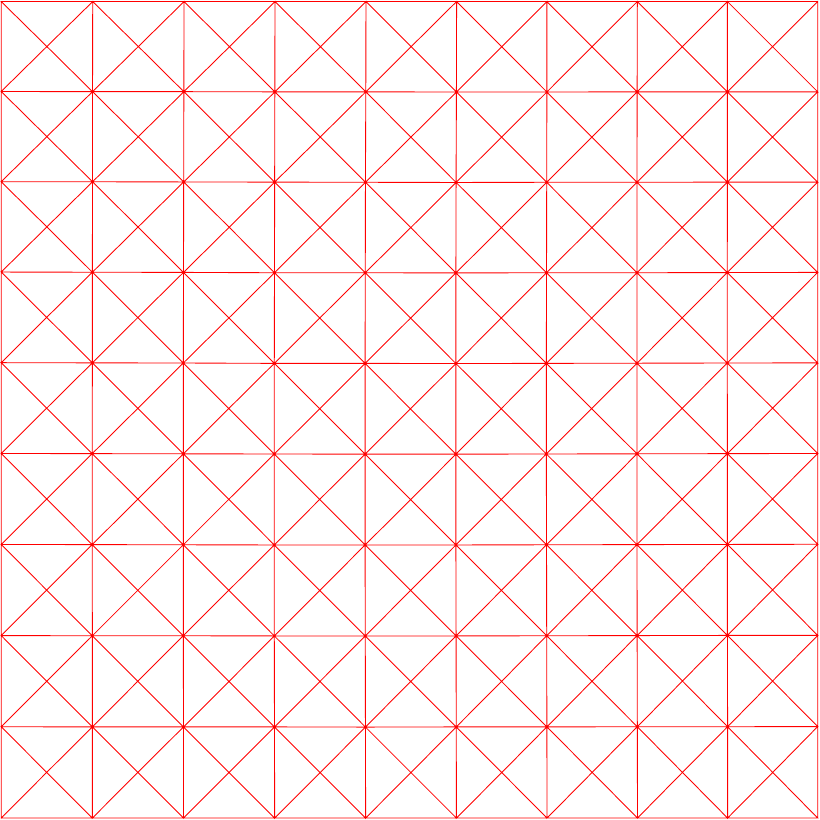}%
      \caption{final, smoothed physical mesh}\label{fig:smoothingU:smoothed}
   \end{subfigure}%
   \\[1ex]
   \begin{subfigure}[t]{0.31\linewidth}
      \includegraphics[clip, width=1.0\linewidth]{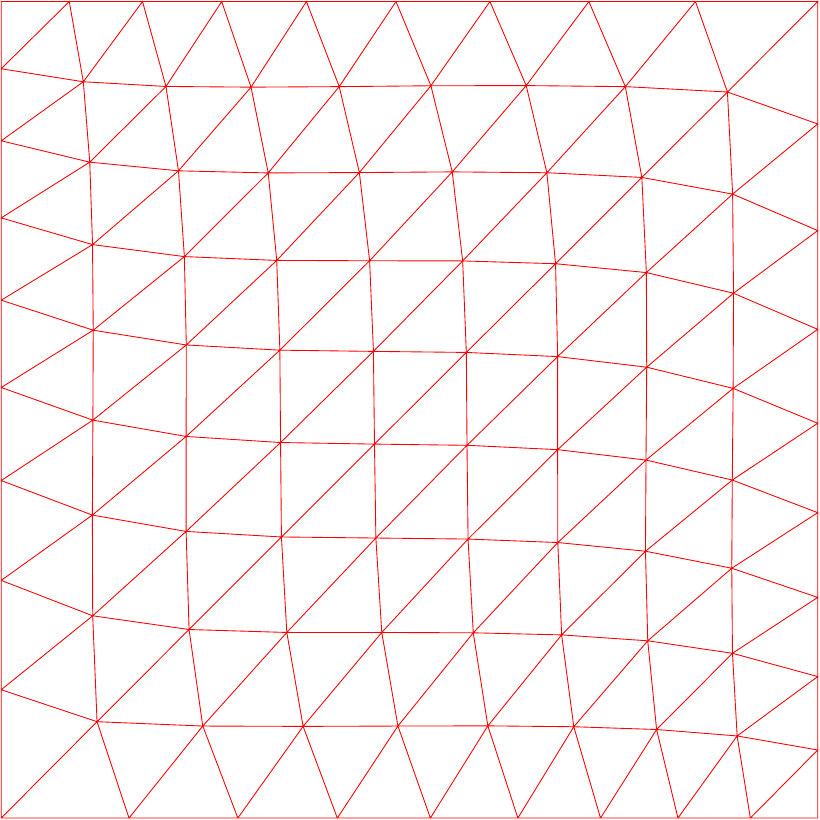}%
      \caption{another $\ThcO$}\label{fig:smoothingR:regular}
   \end{subfigure}%
   \hfill%
   \begin{subfigure}[t]{0.31\linewidth}
      \includegraphics[clip, width=1.0\linewidth]{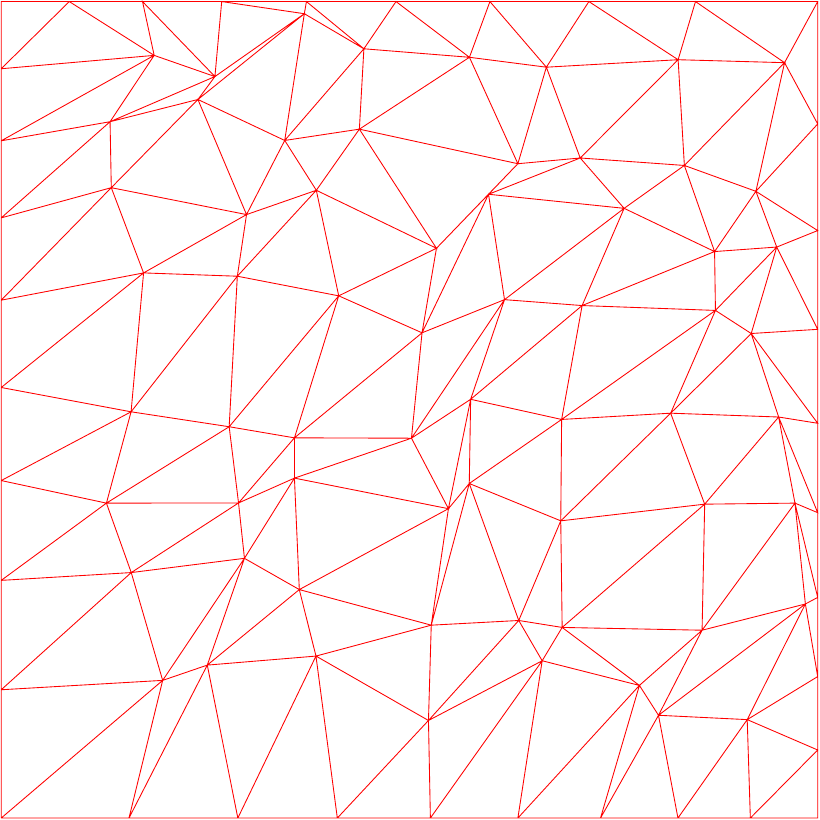}%
      \caption{initial, perturbed physical mesh}\label{fig:smoothingR:perturbed}
   \end{subfigure}%
   \hfill%
   \begin{subfigure}[t]{0.31\linewidth}
      \includegraphics[clip, width=1.0\linewidth]{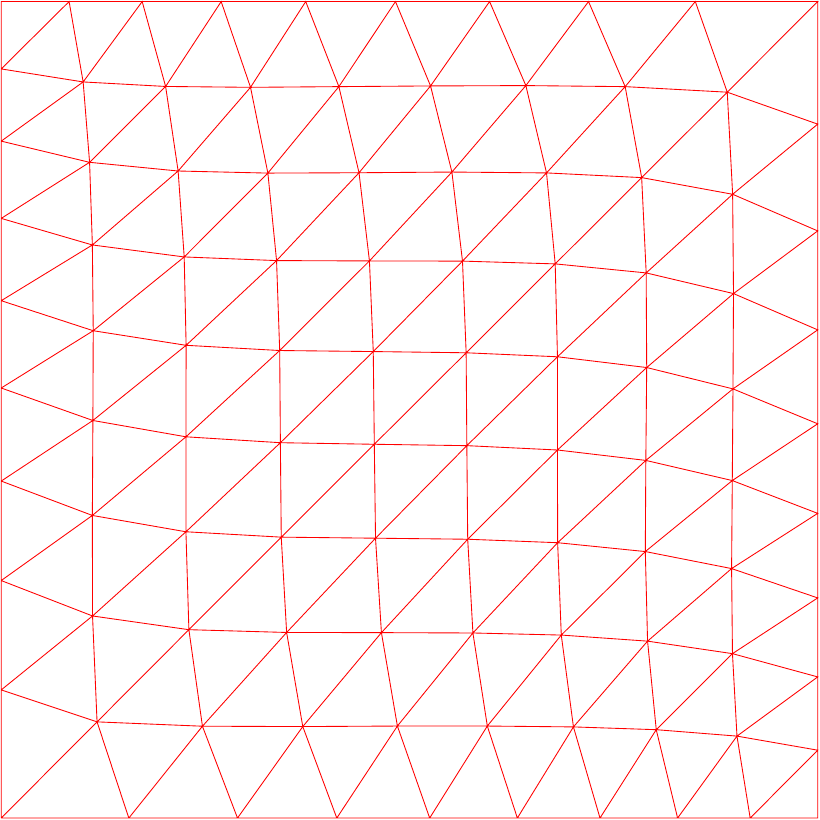}%
      \caption{final, smoothed physical mesh}\label{fig:smoothingR:smoothed}
   \end{subfigure}
   \caption{Mesh examples for mesh smoothing in \cref{ex:smoothing}\label{fig:smoothing}}
\end{figure}
\end{example}

\begin{example}[2D, sine wave]
\label{ex:2d:sin}
In this example we generate an adaptive mesh to minimize the $L^2$ interpolation error bound for
\[
   u(\V{x}) = \tanh \left(
      -30 \left[ y - \frac{1}{2} - \frac{1}{4} \sin\left(2 \pi x\right) \right] \right)
\]
in the domain $\Omega = (0,1) \times (0,1)$ (see~\cite[Sect.~3.2]{Hua02b} on the choice of $\M$).
\Cref{fig:2d:sin} shows the adaptive mesh and close-ins at the tip of the sine wave and in the middle of the domain.

\begin{figure}[p]
   \begin{subfigure}[t]{0.31\linewidth}
      \includegraphics[clip, width=1.0\linewidth]{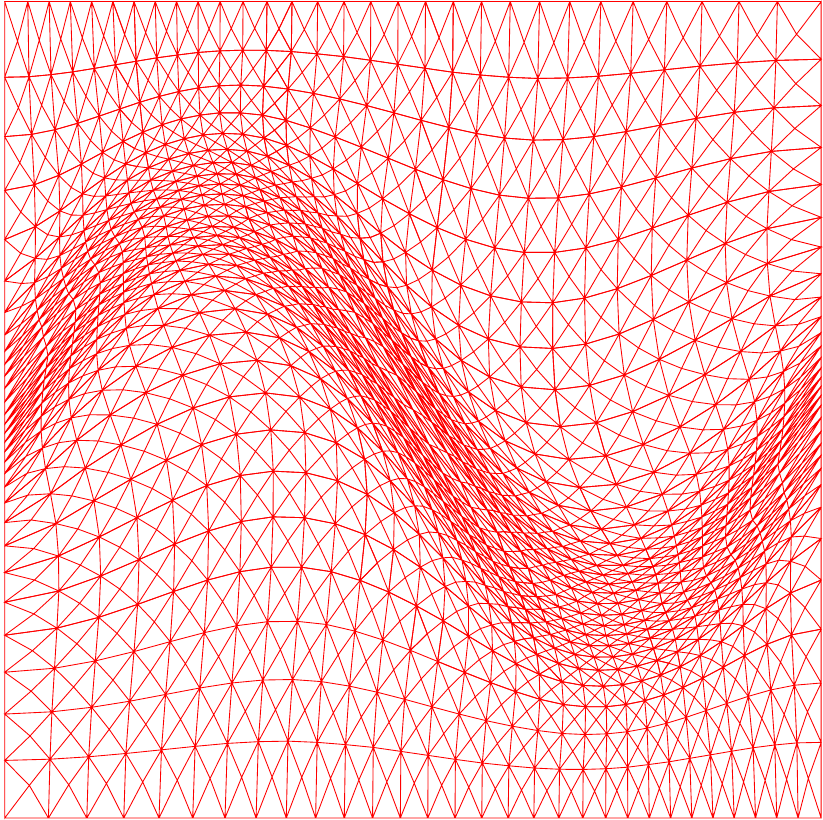}%
      \caption{full domain}
   \end{subfigure}%
   \hfill%
   \begin{subfigure}[t]{0.31\linewidth}
      \includegraphics[clip, width=1.0\linewidth]{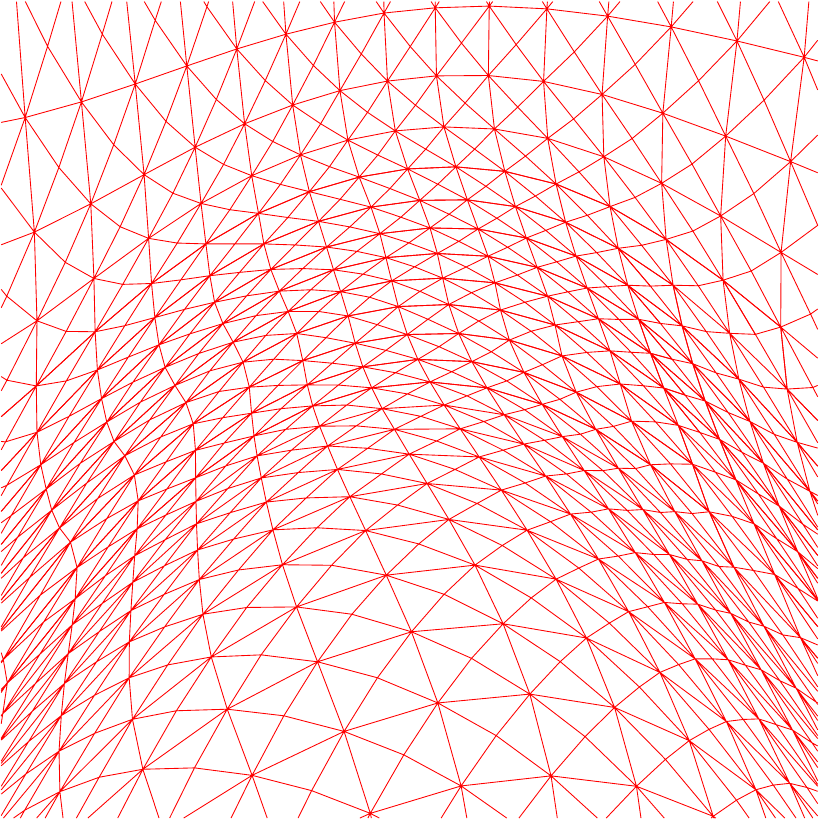}%
      \caption{zoom at the wave tip}
   \end{subfigure}%
   \hfill%
   \begin{subfigure}[t]{0.31\linewidth}
      \includegraphics[clip, width=1.0\linewidth]{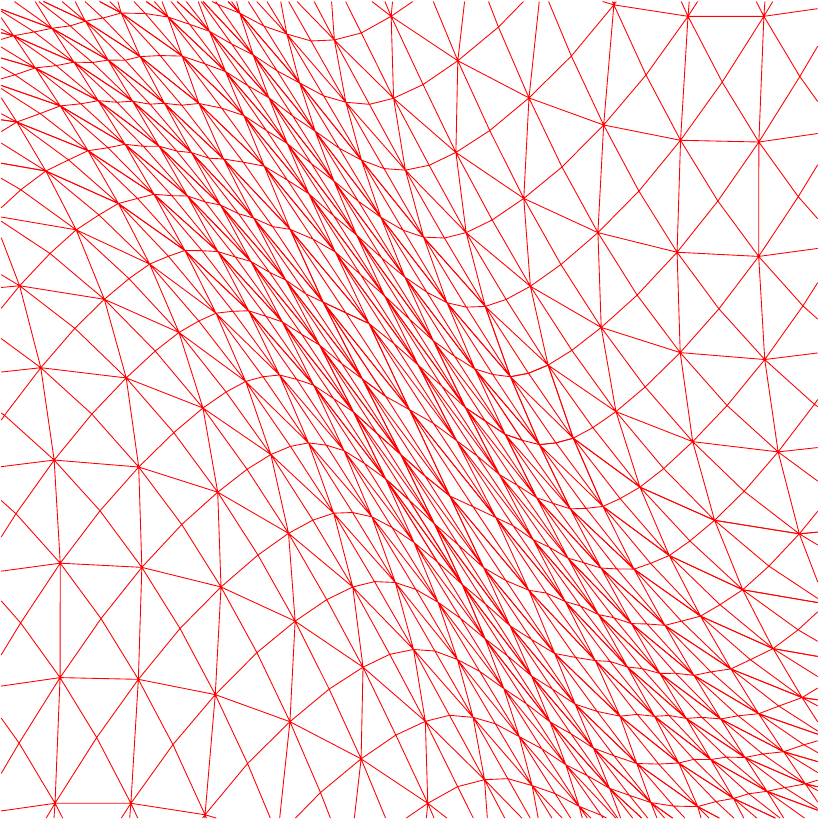}%
      \caption{zoom in the middle}
   \end{subfigure}
   \caption{%
      A $31 \times 31$ adaptive mesh for \cref{ex:2d:sin}
   }\label{fig:2d:sin}
\end{figure}

\end{example}

\begin{example}[2D, horseshoe domain]
\label{ex:horse}

In this example we have different computational and physical domains:$\Omega_c = (0,1) \times (0,1)$ and $\Omega$ is given through the parametrization  
\[
   \V{x} = \begin{pmatrix} 
      - \left(1+\eta \right) \cos\left(\pi \xi \right)\\
        \left[1 + \left(2R - 1\right) \eta \right] \sin\left(\pi\xi \right)
      \end{pmatrix}
      \quad \text{with} \quad
      R = 4.5.
\]
As metric tensors we consider $\M = I$ and
\begin{equation}
   \M(\V{x}) = 1 + {\left( x_1^2 + \sqrt{ {(x_2 - 2 R)}^2 + \num{1e-08} } \right)}^{-1}
   .
   \label{eq:horse:M}
\end{equation}
\Cref{fig:horse:huang} shows initial and adaptive physical meshes with $\M = I$ and $\M$ from \cref{eq:horse:M} for the uniform $21 \times 21$ criss-cross computational mesh of $\Omega_c$.

For this example, it is known that an improper discretization of Winslow's functional can produce folded meshes~\cite{KL95b}.
In \cref{fig:horse:I:winslow,fig:horse:I:huang} we show mesh examples for Winslow's and Huang's functionals with $\M=I$ and different uniform computational meshes ($3 \times 3$, $5 \times 5$ and $9 \times 9$).
Notice that there is no mesh folding even on the coarsest meshes whereas other discretization methods can lead to mesh folding even on much finer meshes (cf.~\cite[Table~1]{KL95b}).

If the highly adaptive $\M$ from \cref{eq:horse:M} is used, the situation becomes different and
mesh folding does occur on a coarse mesh level.
For Huang's functional the critical mesh size is about $17 \times 17$ (\cref{fig:horse:A:huang}), whereas the critical mesh size for Winslow's functional is about $125 \times 125$ (\cref{fig:horse:A:winslow}).

\begin{figure}[t]
   \begin{subfigure}[t]{0.31\linewidth}
      \includegraphics[clip, width=1.0\linewidth]{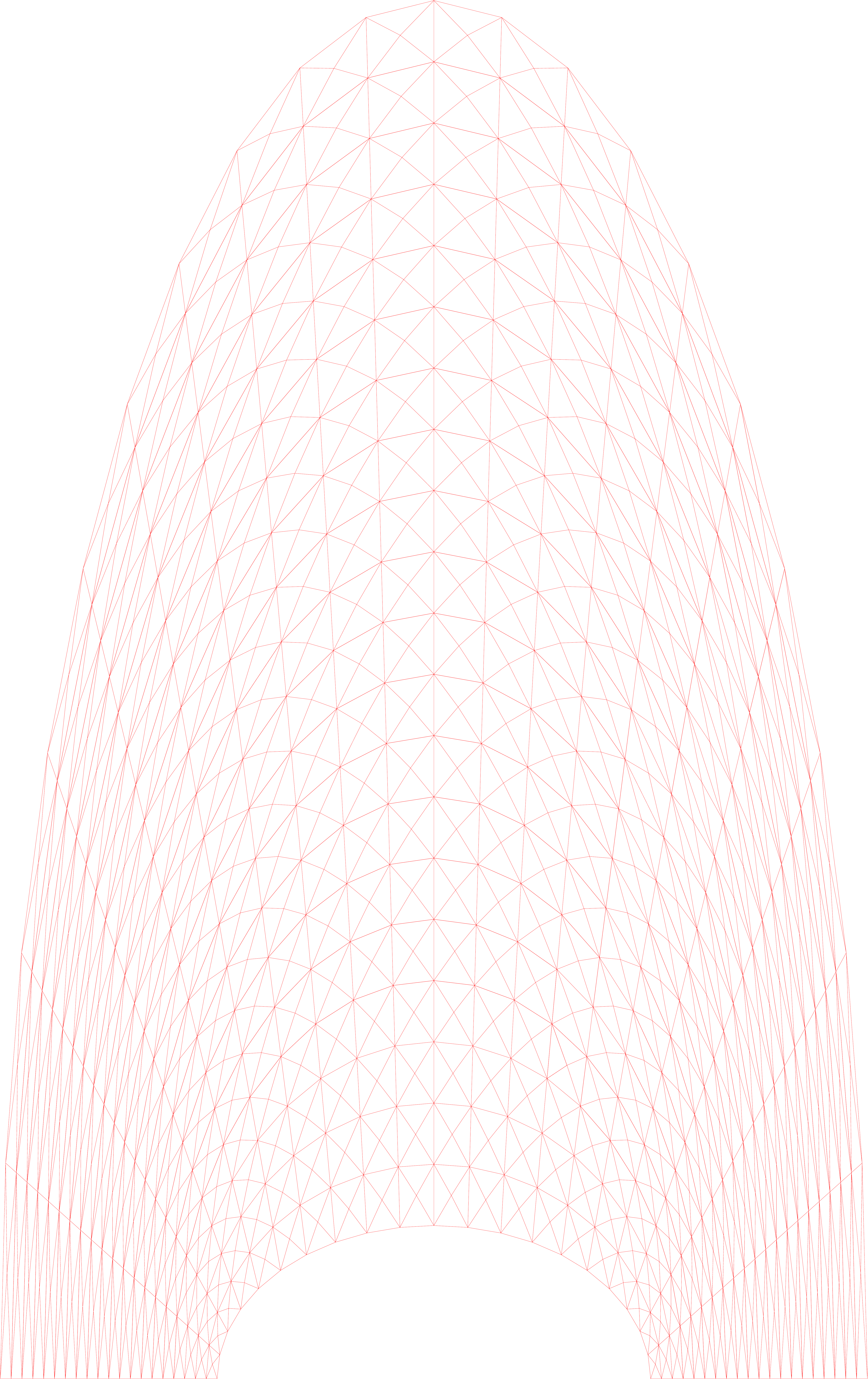}%
      \caption{initial $21 \times 21$ mesh}
   \end{subfigure}%
   \hfill%
   \begin{subfigure}[t]{0.31\linewidth}
      \includegraphics[clip, width=1.0\linewidth]{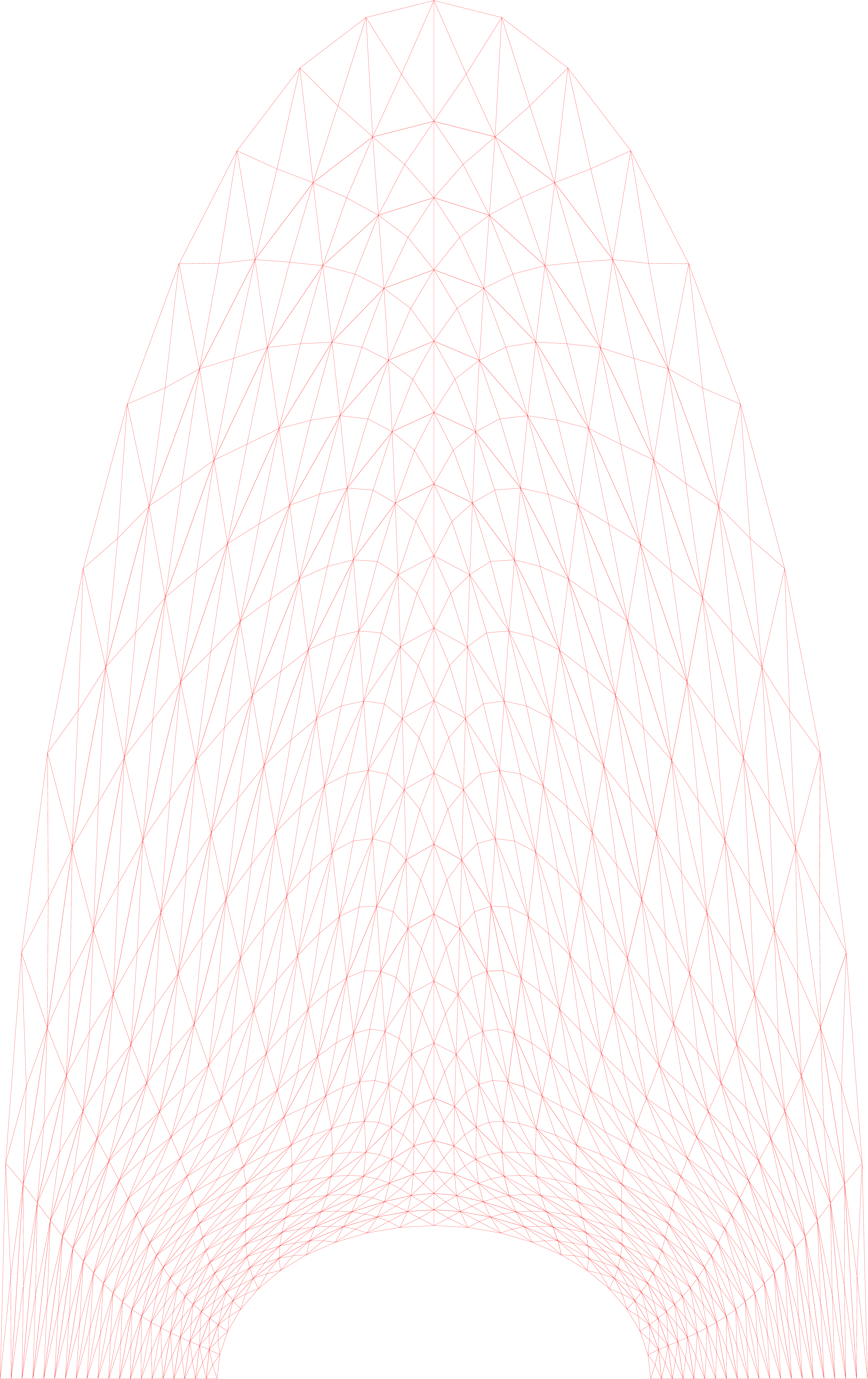}%
      \caption{$\M = I$}
   \end{subfigure}%
   \hfill%
   \begin{subfigure}[t]{0.31\linewidth}
      \includegraphics[clip, width=1.0\linewidth]{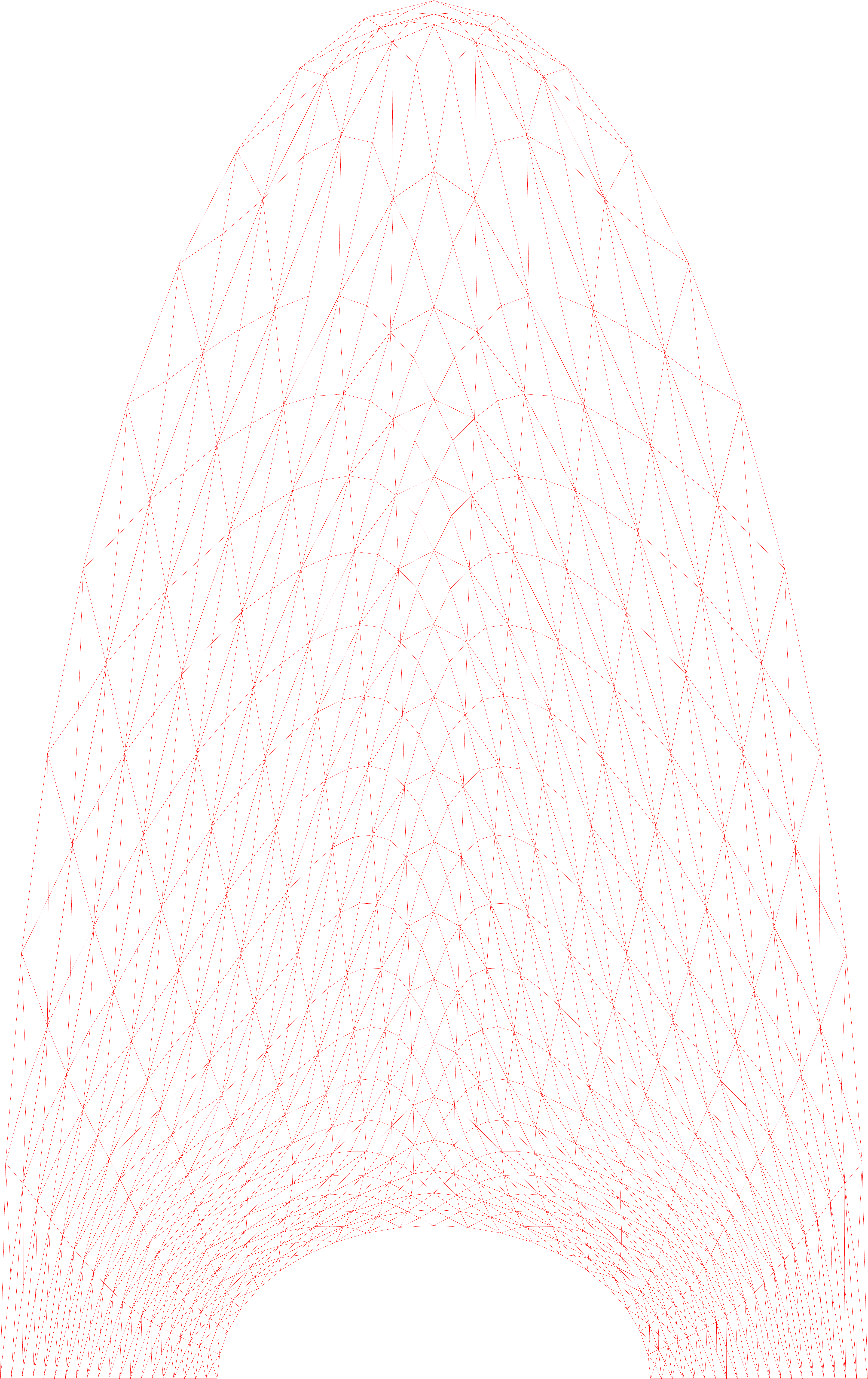}%
      \caption{adaptive $\M$ from \cref{eq:horse:M}}
   \end{subfigure}
   \caption{%
      Mesh examples for \cref{ex:horse}
   }\label{fig:horse:huang}
\end{figure}

\begin{figure}[p]
   \begin{subfigure}[t]{0.31\linewidth}
      \includegraphics[clip, width=1.0\linewidth]{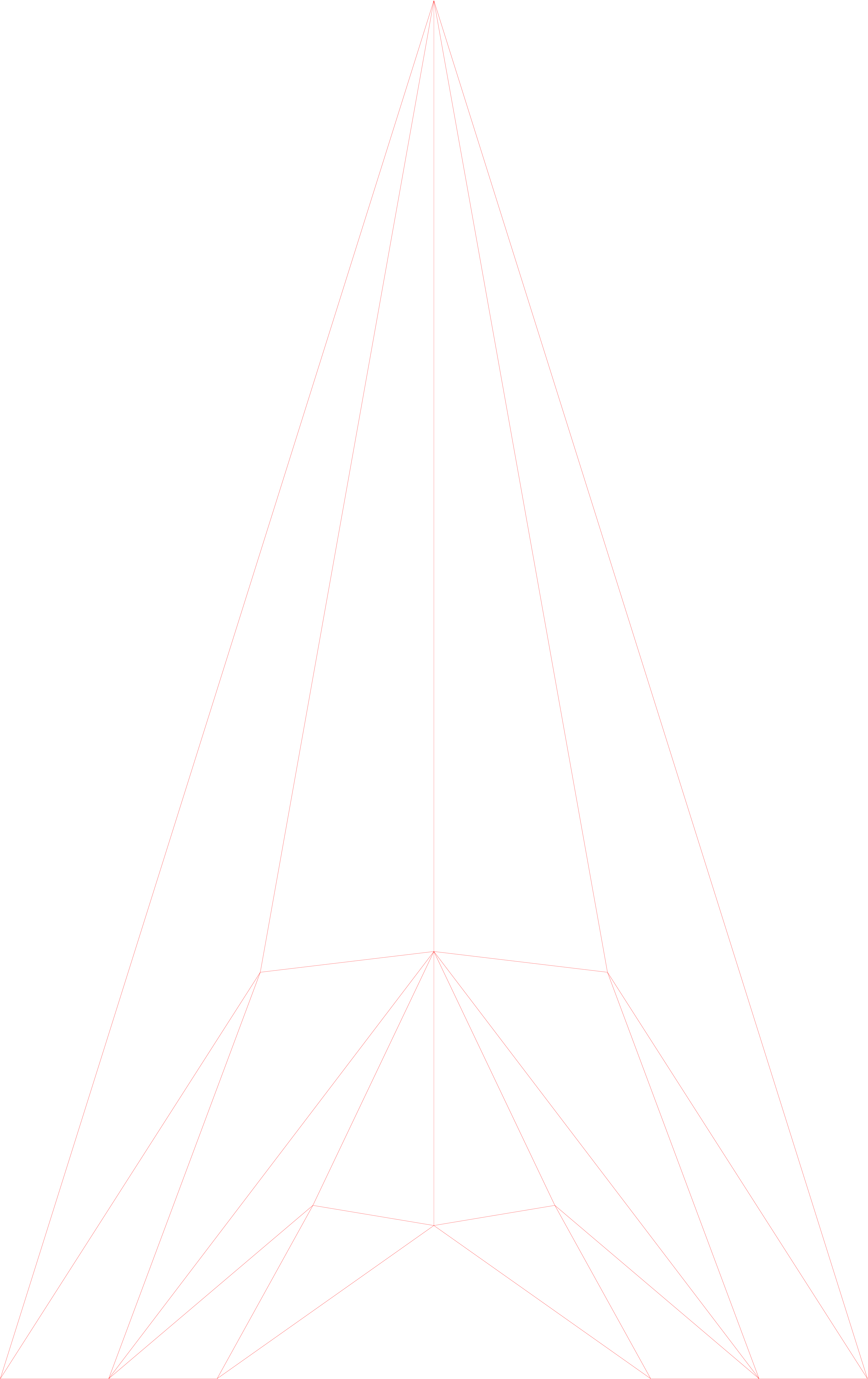}%
      \caption{$3 \times 3$}
   \end{subfigure}%
   \hfill%
   \begin{subfigure}[t]{0.31\linewidth}
      \includegraphics[clip, width=1.0\linewidth]{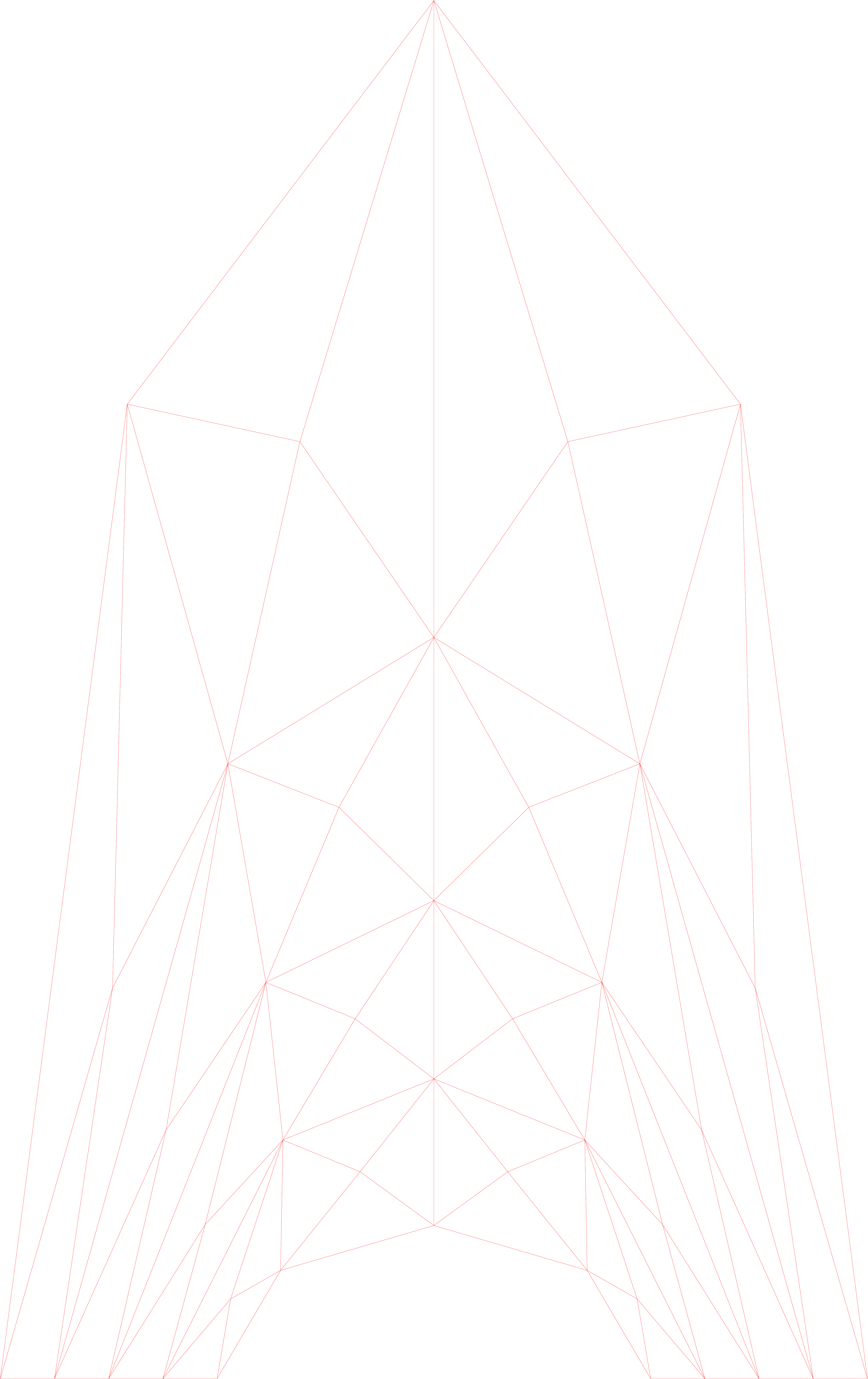}%
      \caption{$5 \times 5$}
   \end{subfigure}%
   \hfill%
   \begin{subfigure}[t]{0.31\linewidth}
      \includegraphics[clip, width=1.0\linewidth]{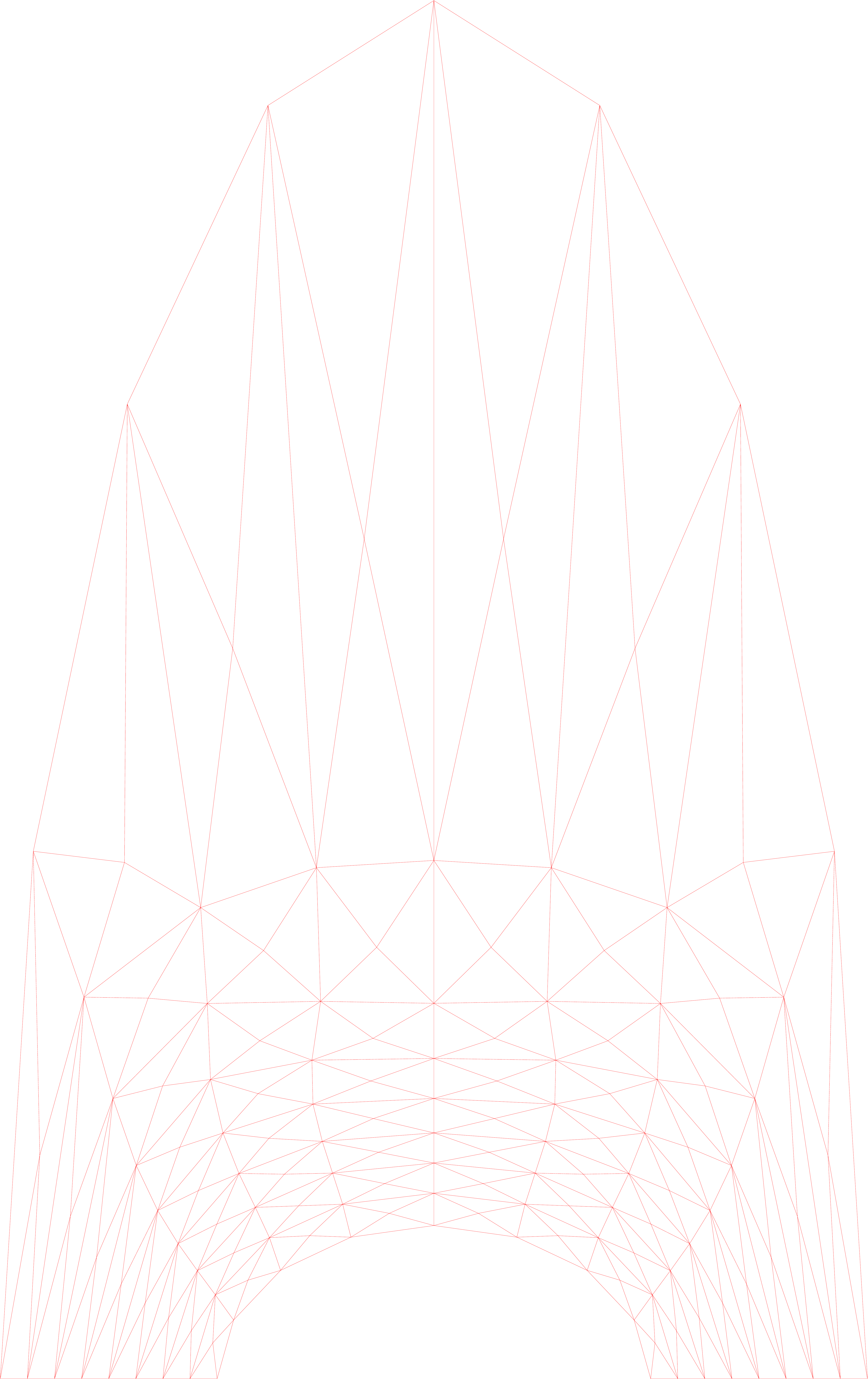}%
      \caption{$9 \times 9$}
   \end{subfigure}
   \caption{%
      \cref{ex:horse},
      $\M = I$,
      Huang's functional \cref{eq:fun:huang}
   }\label{fig:horse:I:huang}
\end{figure}

\begin{figure}[p]
   \begin{subfigure}[t]{0.31\linewidth}
      \includegraphics[clip, width=1.0\linewidth]{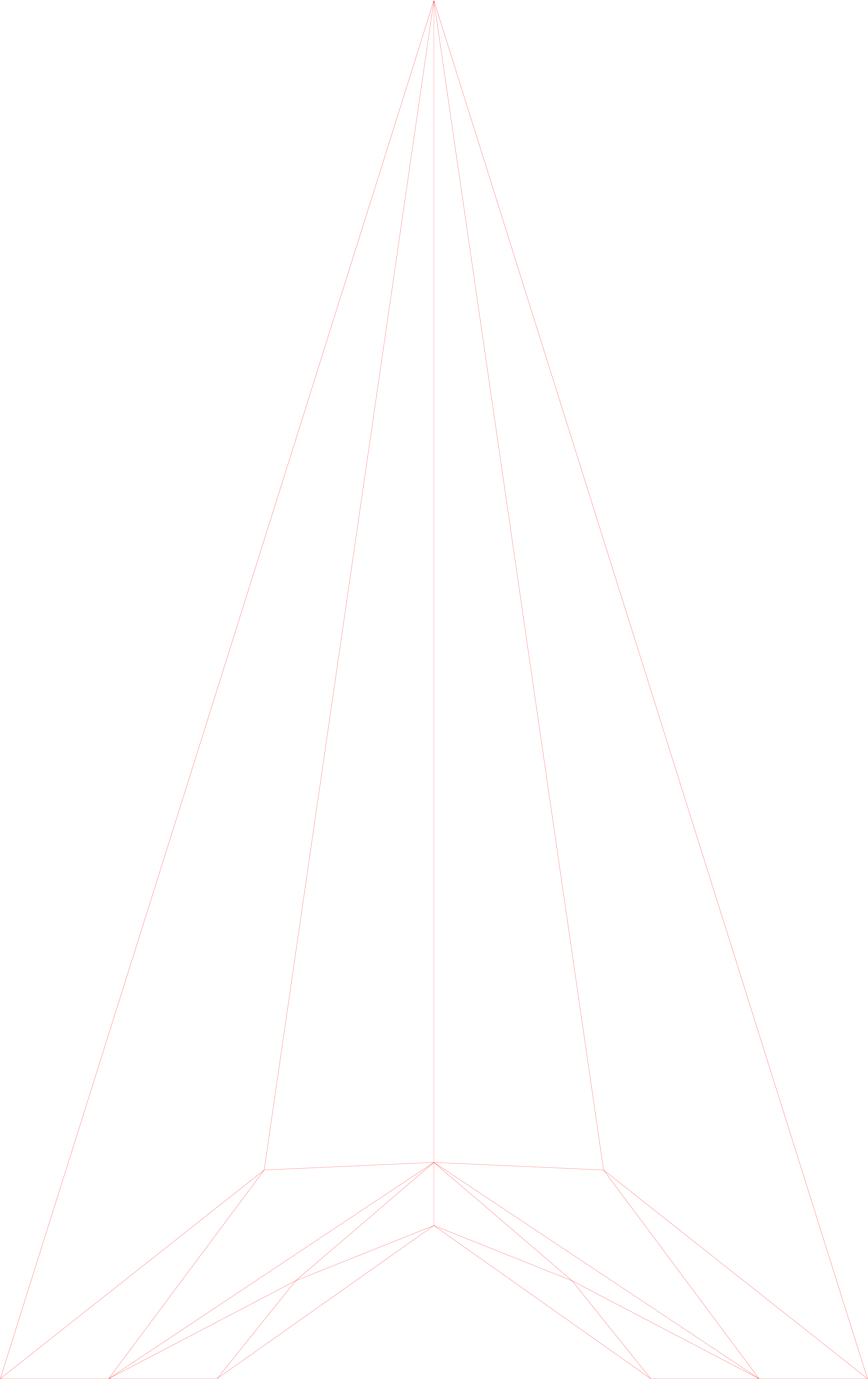}%
      \caption{$3 \times 3$}
   \end{subfigure}%
   \hfill%
   \begin{subfigure}[t]{0.31\linewidth}
      \includegraphics[clip, width=1.0\linewidth]{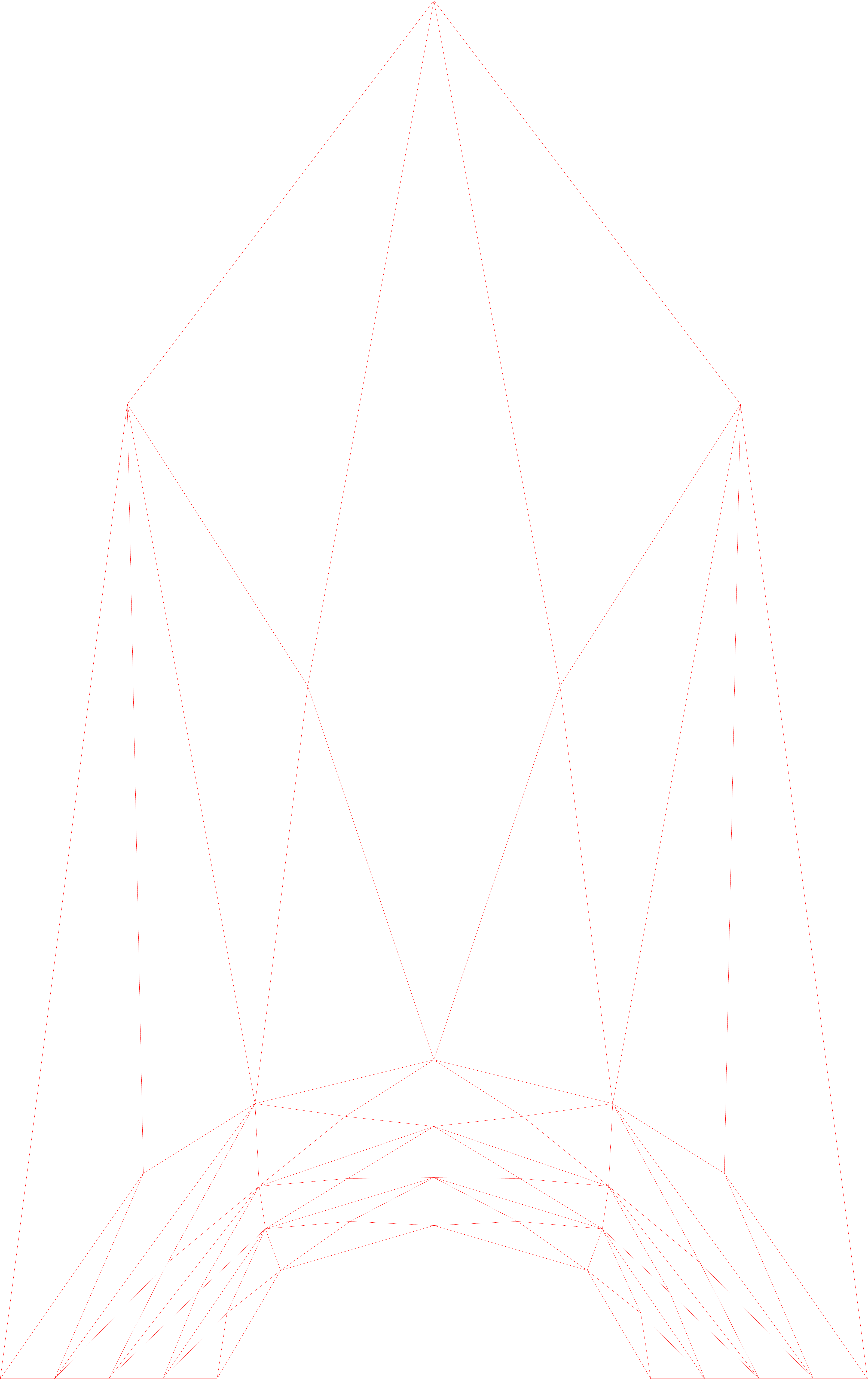}%
      \caption{$5 \times 5$}
   \end{subfigure}%
   \hfill%
   \begin{subfigure}[t]{0.31\linewidth}
      \includegraphics[clip, width=1.0\linewidth]{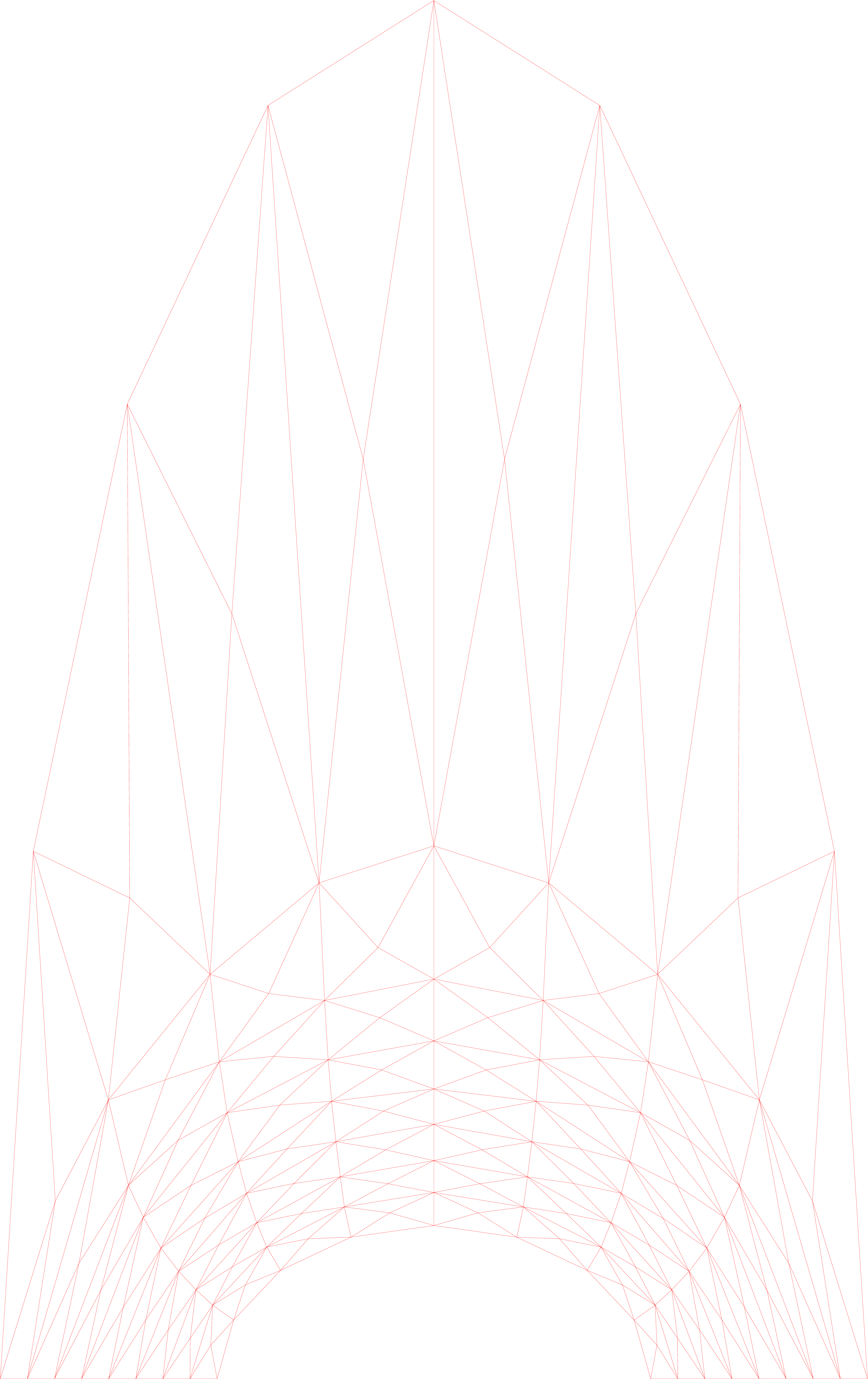}%
      \caption{$9 \times 9$}
   \end{subfigure}
   \caption{%
      \Cref{ex:horse},
      $\M = I$,
      Winslow's functional \cref{eq:fun:winslow}
   }\label{fig:horse:I:winslow}
\end{figure}

\begin{figure}[p]
   \begin{subfigure}[t]{0.31\linewidth}
      \includegraphics[clip, width=1.0\linewidth]{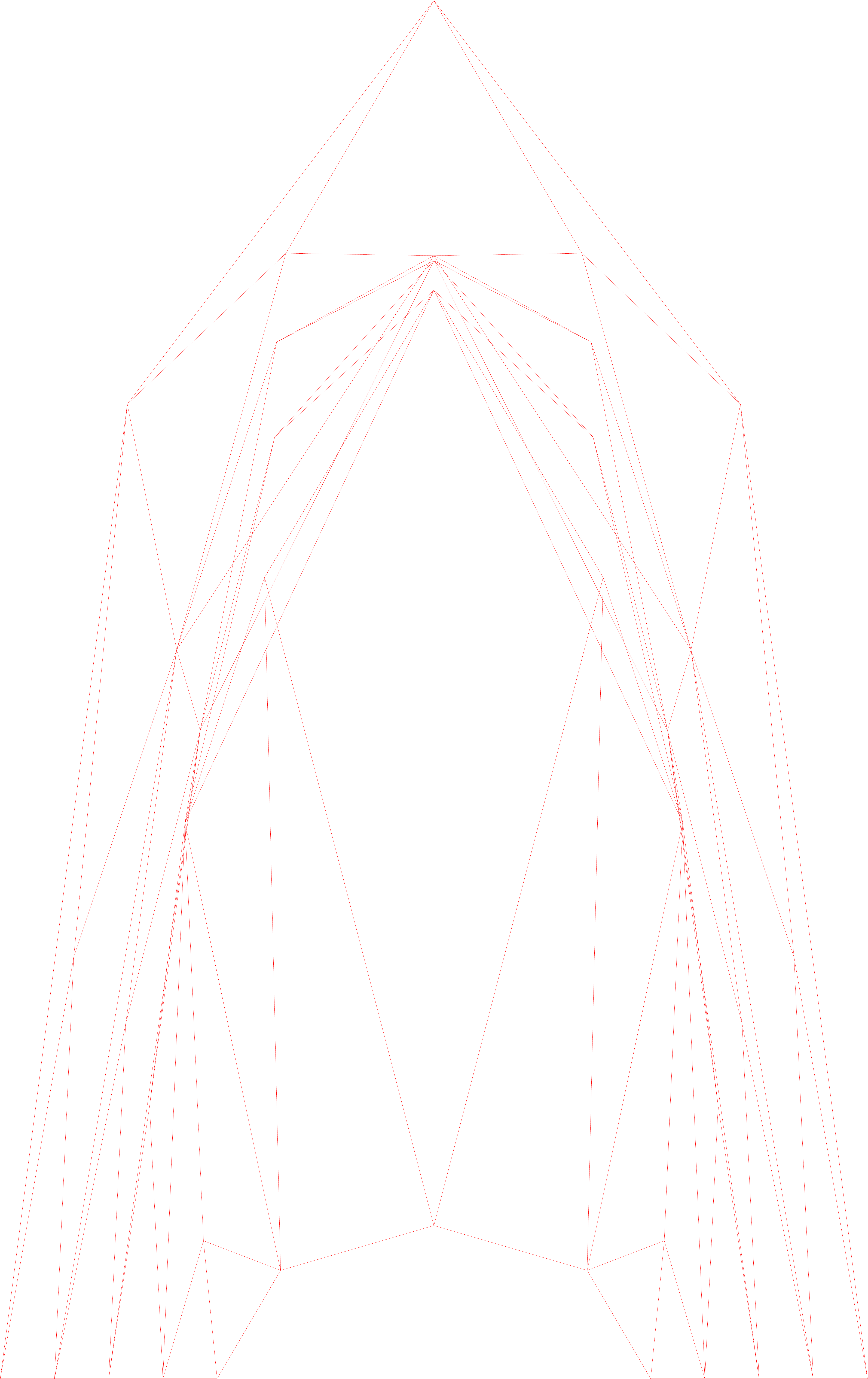}%
      \caption{$5 \times 5$}
   \end{subfigure}%
   \hfill%
   \begin{subfigure}[t]{0.31\linewidth}
      \includegraphics[clip, width=1.0\linewidth]{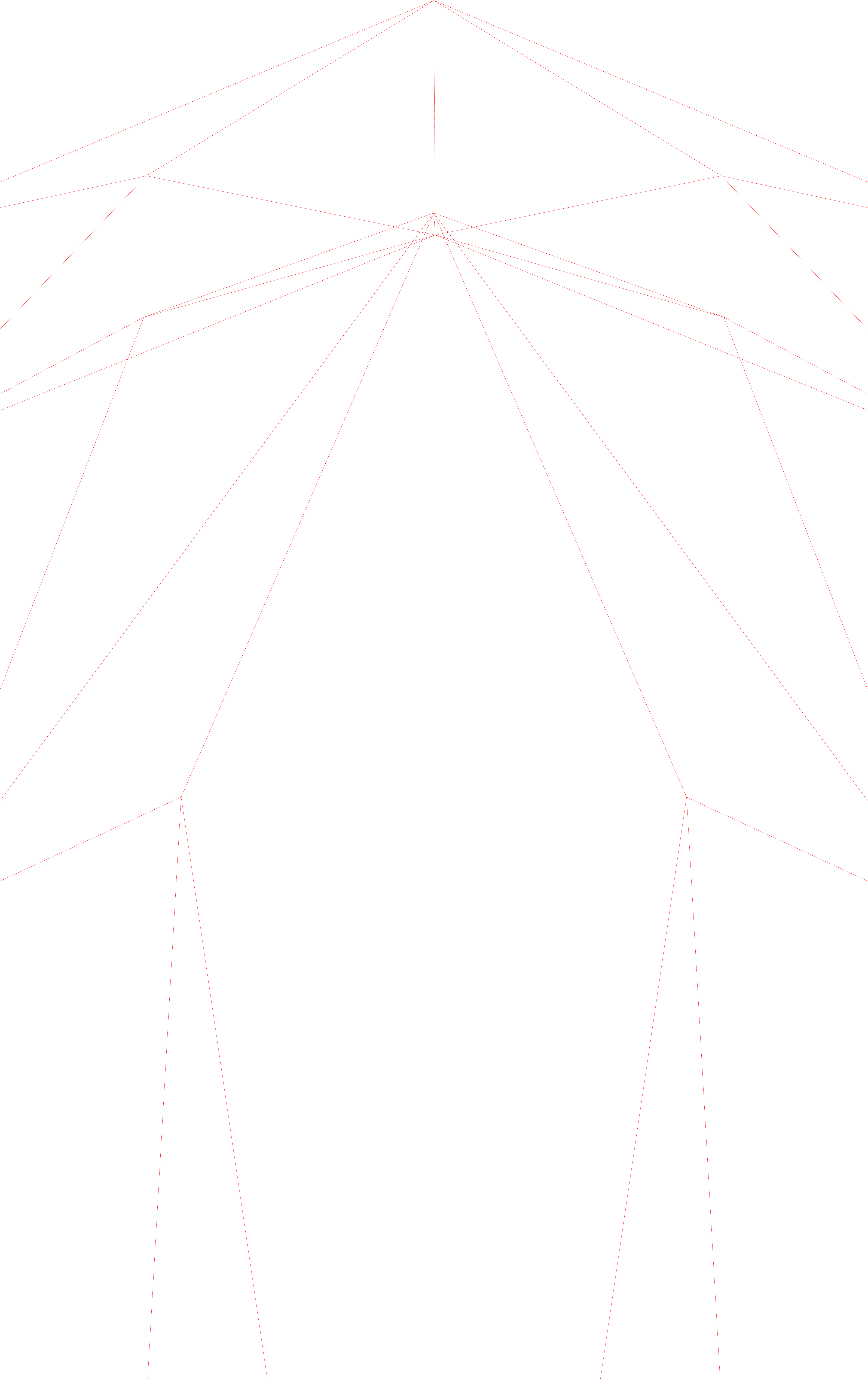}%
      \caption{$13 \times 13$ (zoom at the tip)}
   \end{subfigure}%
   \hfill%
   \begin{subfigure}[t]{0.31\linewidth}
      \includegraphics[clip, width=1.0\linewidth]{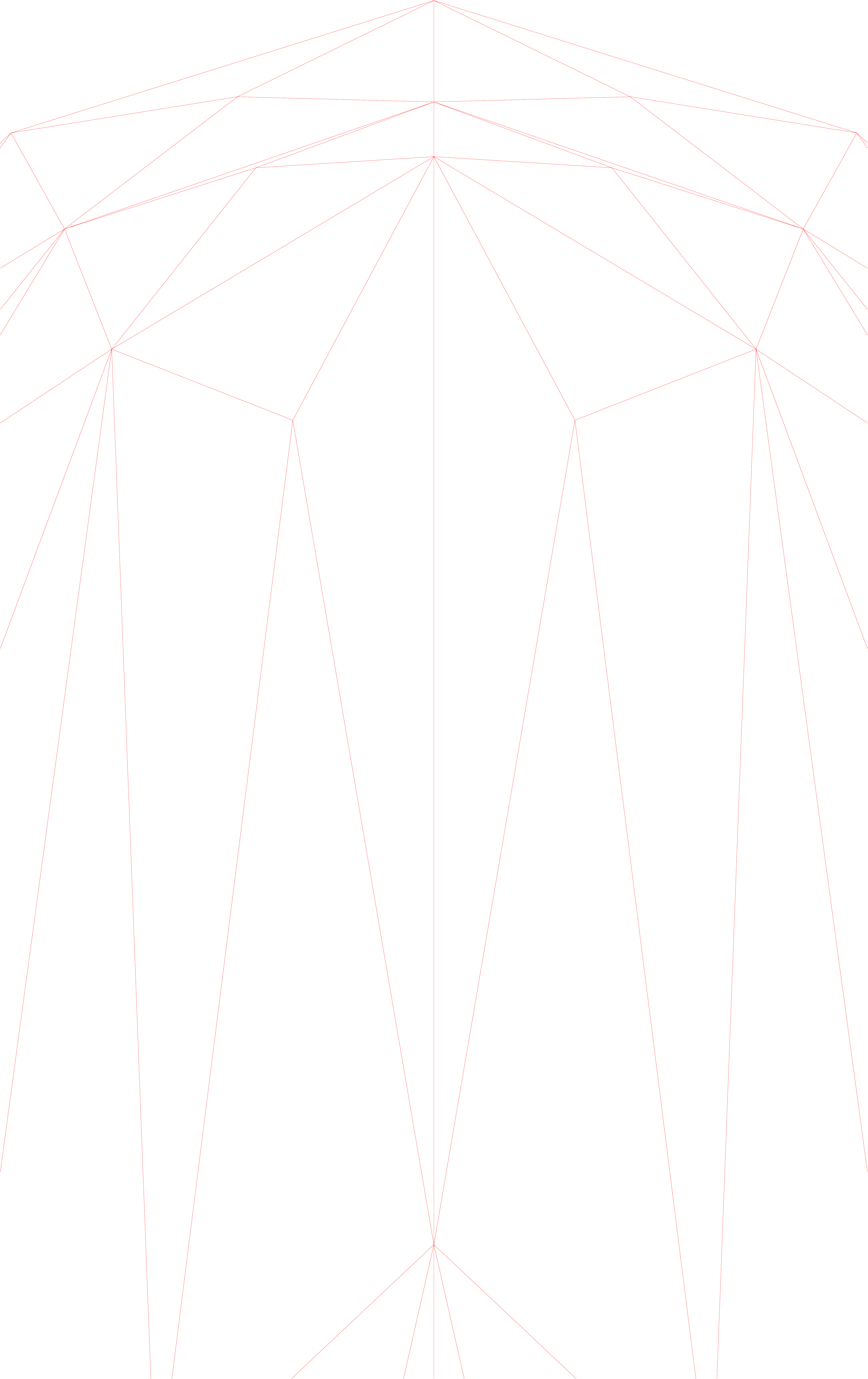}%
      \caption{$17 \times 17$ (zoom at the tip)}
   \end{subfigure}
   \caption{%
      \cref{ex:horse},
      adaptive $\M$ from \cref{eq:horse:M},
      Huang's functional \cref{eq:fun:huang}
   }\label{fig:horse:A:huang}
\end{figure}

\begin{figure}[p]
   \begin{subfigure}[t]{0.31\linewidth}
      \includegraphics[clip, width=1.0\linewidth]{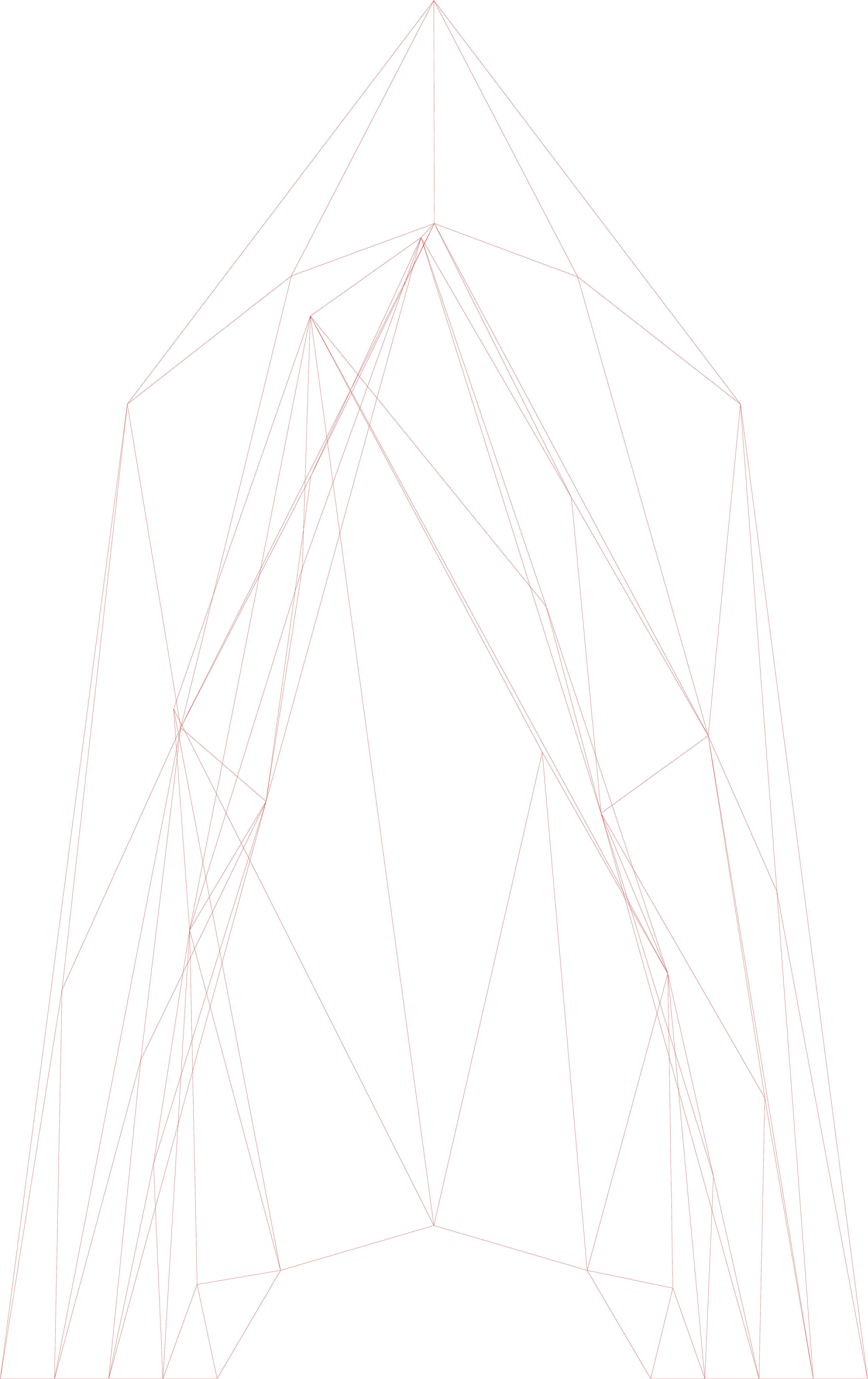}%
      \caption{$5 \times 5$}
   \end{subfigure}%
   \hfill%
   \begin{subfigure}[t]{0.31\linewidth}
      \includegraphics[clip, width=1.0\linewidth]{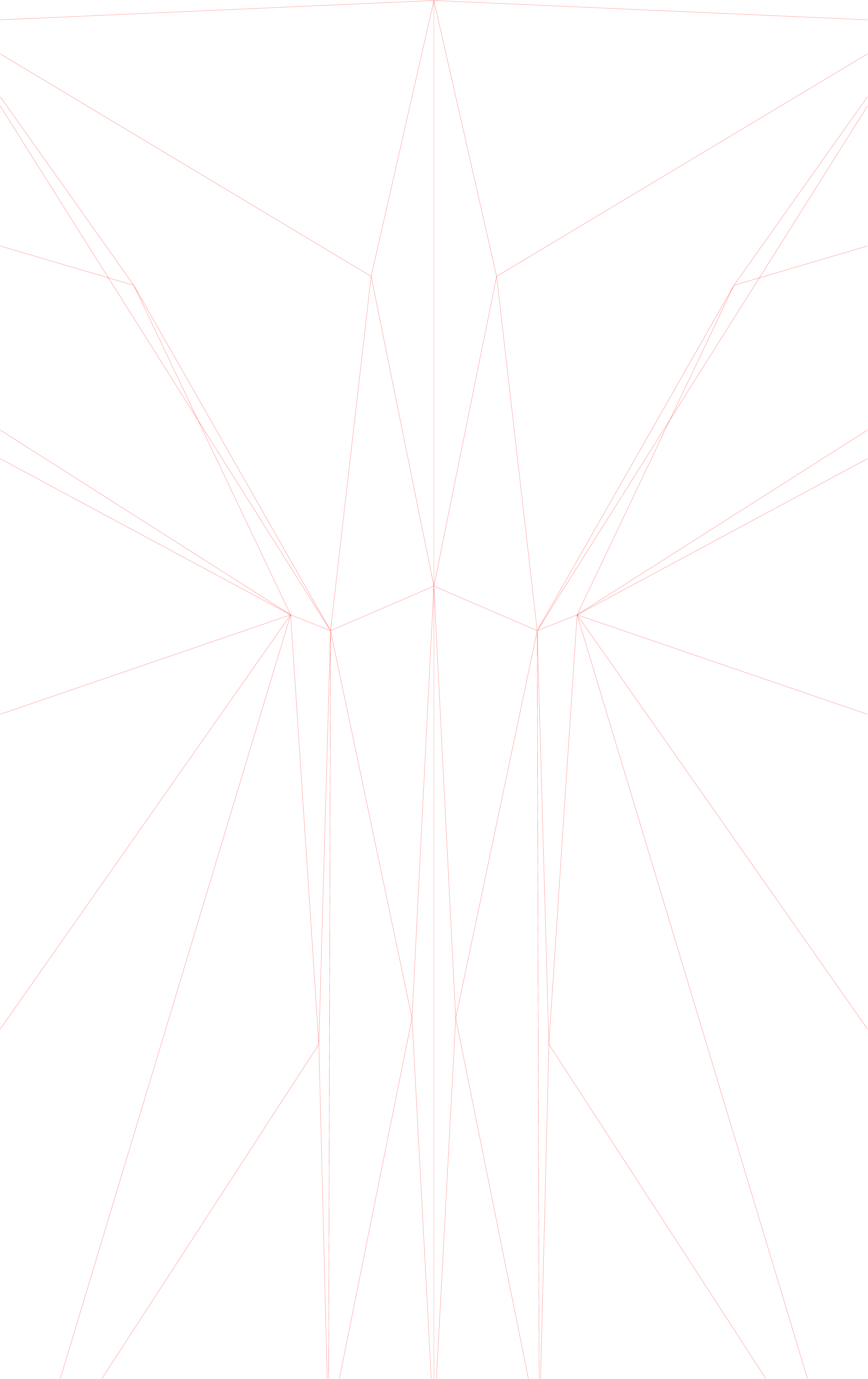}%
      \caption{$113 \times 113$ (zoom at the tip)}
   \end{subfigure}%
   \hfill%
   \begin{subfigure}[t]{0.31\linewidth}
      \includegraphics[clip, width=1.0\linewidth]{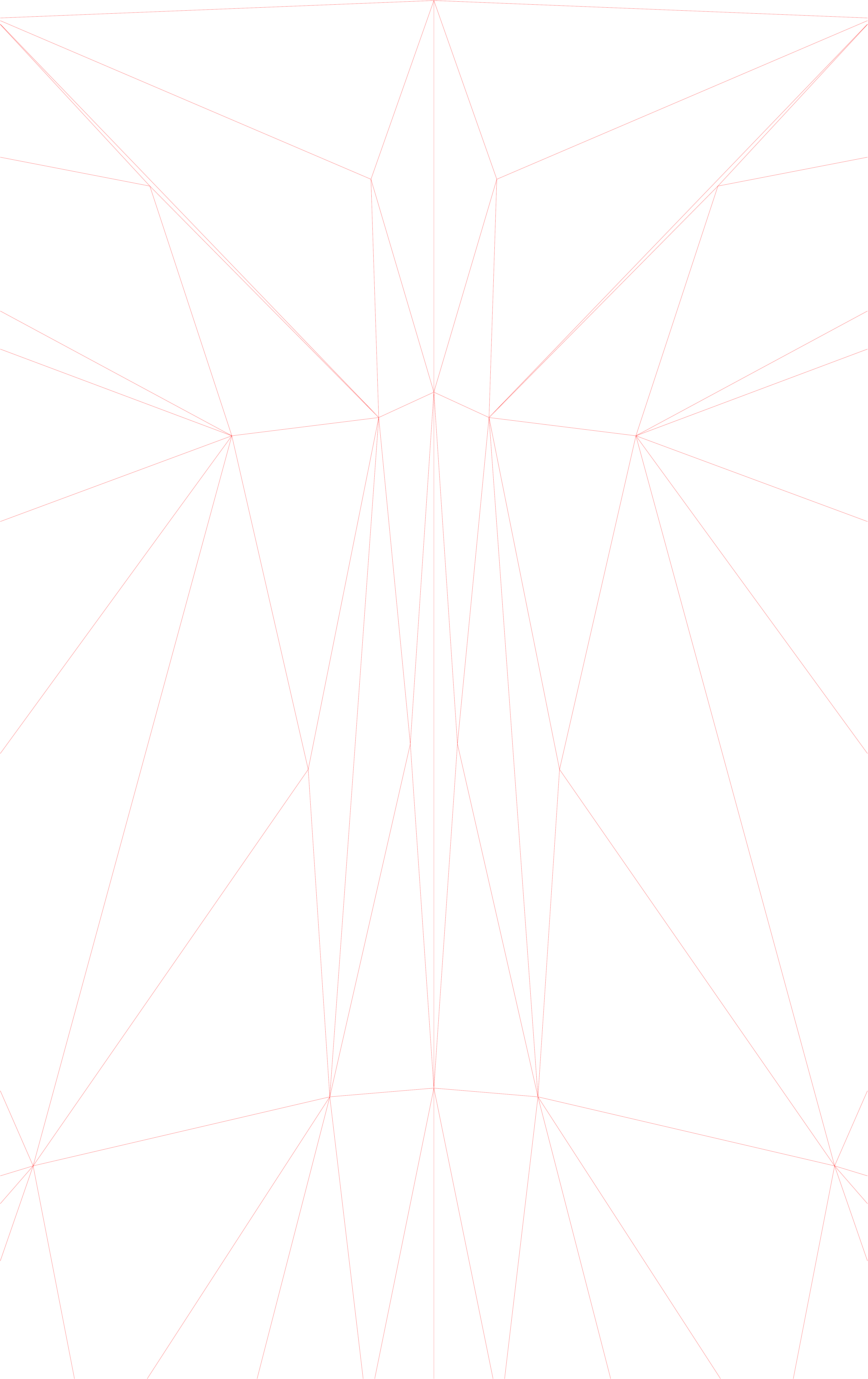}%
      \caption{$125 \times 125$ (zoom at the tip)}
   \end{subfigure}
   \caption{%
      \Cref{ex:horse},
      adaptive $\M$ from \cref{eq:horse:M},
      Winslow's functional \cref{eq:fun:winslow}
   }\label{fig:horse:A:winslow}
\end{figure}%
\end{example}

\begin{example}[3D, nine spheres]
\label{ex:nine:spheres}
In this example we choose $\Omega = (-2,2) \times (-2,2) \times (-2,2)$ and $\M$ to minimize the $L^2$ interpolation error bound (see~\cite{Hua02b} on the choice of $\M$) for
\begin{align*}
   u(\V{x}) &=
      \tanh\left(30\left[{(x-0.0)}^2 + {(y-0.0)}^2 + {(z-0.0)}^2 - 0.1875\right]\right)\\
   &+ \tanh\left(30\left[{(x-0.5)}^2 + {(y-0.5)}^2 + {(z-0.5)}^2 - 0.1875\right]\right)\\
   &+ \tanh\left(30\left[{(x-0.5)}^2 + {(y+0.5)}^2 + {(z-0.5)}^2 - 0.1875\right]\right)\\
   &+ \tanh\left(30\left[{(x+0.5)}^2 + {(y-0.5)}^2 + {(z-0.5)}^2 - 0.1875\right]\right)\\
   &+ \tanh\left(30\left[{(x+0.5)}^2 + {(y+0.5)}^2 + {(z-0.5)}^2 - 0.1875\right]\right)\\
   &+ \tanh\left(30\left[{(x-0.5)}^2 + {(y-0.5)}^2 + {(z+0.5)}^2 - 0.1875\right]\right)\\
   &+ \tanh\left(30\left[{(x-0.5)}^2 + {(y+0.5)}^2 + {(z+0.5)}^2 - 0.1875\right]\right)\\
   &+ \tanh\left(30\left[{(x+0.5)}^2 + {(y-0.5)}^2 + {(z+0.5)}^2 - 0.1875\right]\right)\\
   &+ \tanh\left(30\left[{(x+0.5)}^2 + {(y+0.5)}^2 + {(z+0.5)}^2 - 0.1875\right]\right)
   .
\end{align*}

\Cref{fig:nine:spheres} shows an example of an adaptive mesh and cuts through the mesh in $y$-$z$ plane.
\end{example}

\begin{figure}[p]
   \begin{subfigure}[t]{0.37\linewidth}
      \includegraphics[clip, width=1.0\linewidth]{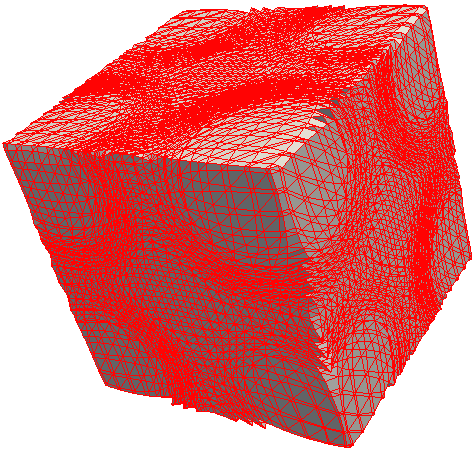}%
      \caption{inside-out cutaway}
   \end{subfigure}%
   \qquad\qquad%
   \begin{subfigure}[t]{0.37\linewidth}
      \includegraphics[clip, width=1.0\linewidth]{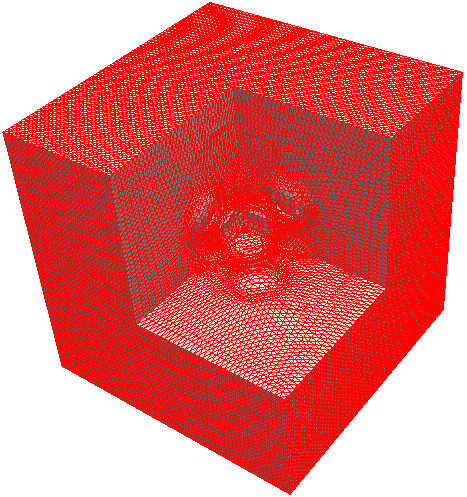}%
      \caption{cutaway}
   \end{subfigure}%
   \\[1ex]
   \begin{subfigure}[t]{0.31\linewidth}
      \includegraphics[clip, width=1.0\linewidth]{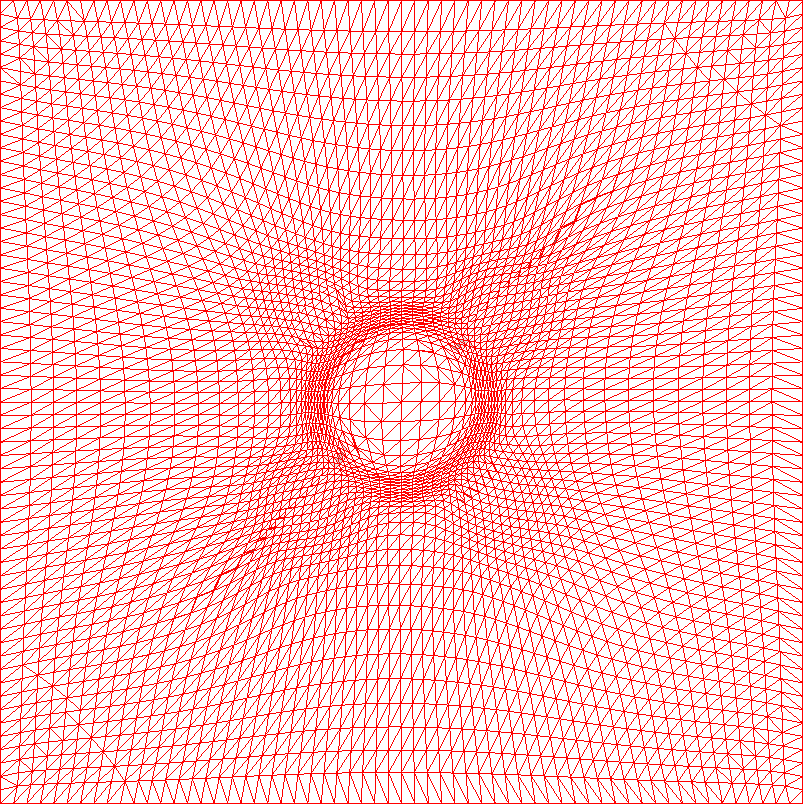}%
      \caption{$y$-$z$ cut at $x = 0$}
   \end{subfigure}%
   \hfill%
   \begin{subfigure}[t]{0.31\linewidth}
      \includegraphics[clip, width=1.0\linewidth]{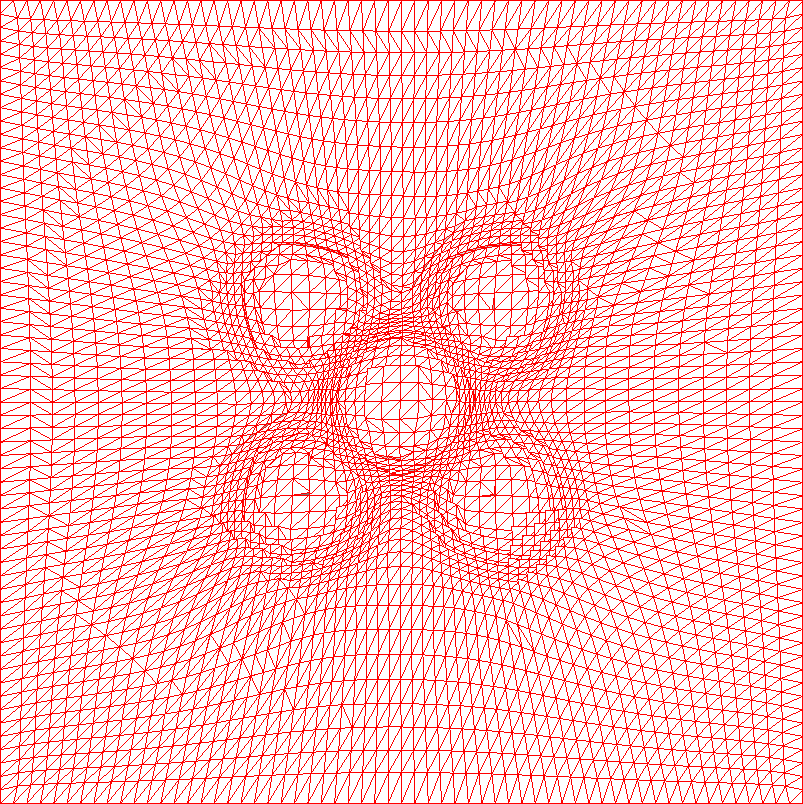}%
      \caption{$y$-$z$ cut at $x = 0.25$}
   \end{subfigure}%
   \hfill%
   \begin{subfigure}[t]{0.31\linewidth}
      \includegraphics[clip, width=1.0\linewidth]{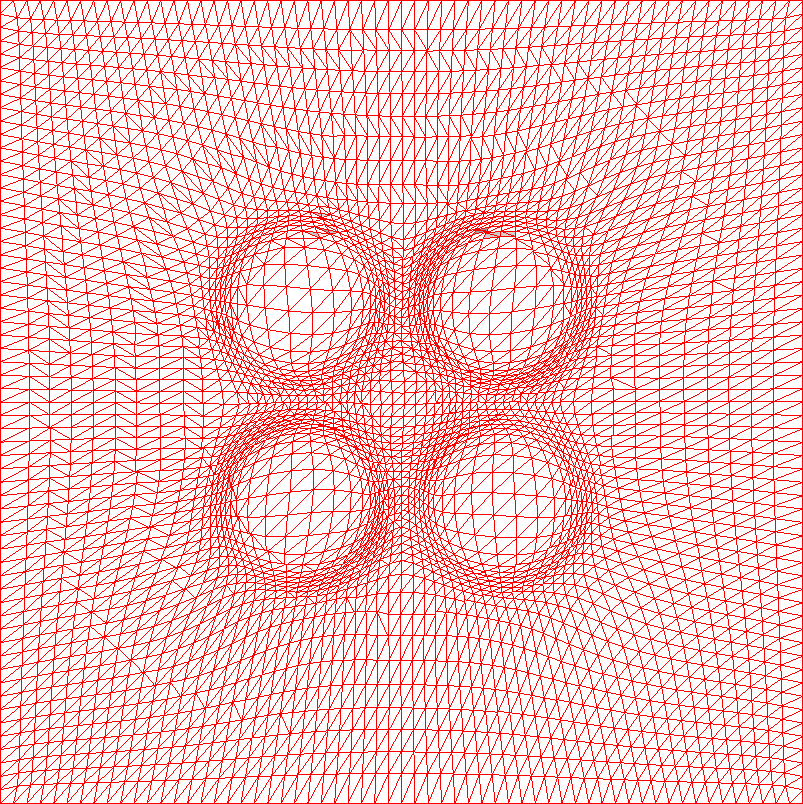}%
      \caption{$y$-$z$ cut at $x = 0.5$}
   \end{subfigure}
   \caption{Adaptive mesh example and $y$-$z$ plane cuts for \cref{ex:nine:spheres}\label{fig:nine:spheres}}
\end{figure}
{
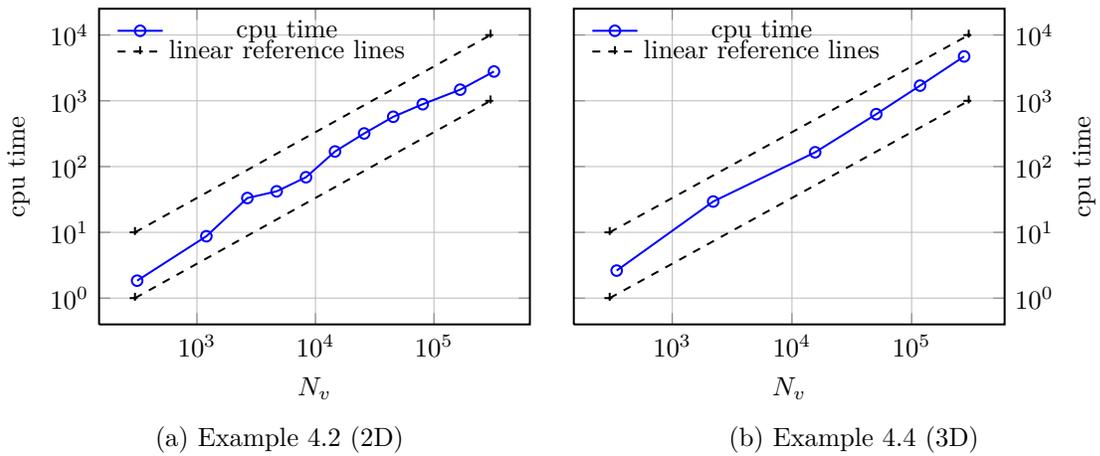
\begin{figure}[p]
   \begin{subfigure}[t]{0.5\linewidth}
      \begin{tikzpicture}
         \begin{loglogaxis}[
            width=0.75\linewidth,%
            height=0.55\linewidth,%
            xlabel={$N_v$},
            ylabel={cpu time},
         ]
            \addplot[blue, solid, mark=o]
               table [x index=0, y index=6, col sep = space]%
               {hk2014-variational-2d-time-iter3.dat};
            \addlegendentry{cpu time}
            \addplot[black, dashed] coordinates {(3.0e+02, 1.0e+01) (3.0e+05, 1.0e+04)};
            \addplot[black, dashed] coordinates {(3.0e+02, 1.0e+00) (3.0e+05, 1.0e+03)};
            \addlegendentry{linear reference lines}
        \end{loglogaxis}
      \end{tikzpicture}
      \caption{\cref{ex:2d:sin} (2D)\label{fig:comp:time:2d}}
   \end{subfigure}%
   \begin{subfigure}[t]{0.5\linewidth}
      \begin{tikzpicture}
         \begin{loglogaxis}[
            width=0.75\linewidth,%
            height=0.55\linewidth,%
            xlabel={$N_v$},
            ylabel={cpu time},
            yticklabel pos=right,
         ]
            \addplot[blue, solid, mark=o]
               table [x index=0, y index=6, col sep = space]%
               {hk2014-variational-3d-time-iter3.dat};
            \addlegendentry{cpu time}
            \addplot[black, dashed] coordinates {(3.0e+02, 1.0e+01) (3.0e+05, 1.0e+04)};
            \addplot[black, dashed] coordinates {(3.0e+02, 1.0e+00) (3.0e+05, 1.0e+03)};
            \addlegendentry{linear reference lines}
        \end{loglogaxis}
      \end{tikzpicture}
      \caption{\cref{ex:nine:spheres} (3D)\label{fig:comp:time:3d}}
   \end{subfigure}
   \caption{CPU time for mesh generation vs.\ number of vertices $N_v$\label{fig:comp:time}}
\end{figure}
}


To test the computational cost of the new method, we measure the CPU time in dependence on the number of mesh nodes for \cref{ex:2d:sin} (2D) and \cref{ex:nine:spheres} (3D) using a simple Matlab implementation running on a system with two Intel Xeons (E5-2665, 2.40 GHz).
Numerical results suggest that the computational cost is linear in the number of mesh nodes for both 2D (\cref{fig:comp:time:2d}) and 3D (\cref{fig:comp:time:3d}) examples; it takes about 10 seconds of CPU time per \num{1000} nodes.

\section{Conclusions and further comments}
\label{sect:conclusion}

In the previous sections we have proposed a direct discretization for a class of functionals used in variational mesh generation and adaptation.
The meshing functionals are discretized directly on simplicial meshes and the Jacobian matrix of the continuous coordinate transformation is approximated by the Jacobian matrices of affine mappings between computational and physical elements (cf.~\cref{eq:dis:fun:1}).
Since the latter are computed from the edge matrices of the elements, there is no need for a direct discretization of derivatives of the coordinate transformation, a daunting task which is typically involved with nonuniform meshes.

Moreover, the proposed discretization has the advantage of preserving the basic geometric structure of the underlying continuous functional.
In particular, it preserves the coercivity and convexity of Winslow's functional and the polyconvexity of Huang's functional (see \cref{sect:preservation}).

A further advantage is the simple analytical formula for derivatives with respect to the coordinates of mesh vertices (cf.~\cref{eq:der:1,eq:der:2}), which allows a simple (and parallel) implementation.
The computational cost is linear in the number of mesh nodes.
A time-varying solution strategy of the mesh equation has also been discussed and applied to a number of numerical examples in \cref{sect:numerics}.
Since variational mesh generation and adaptation is employed as the base for a number of adaptive moving mesh, mesh smoothing, and refinement strategies, the result of this work can be used to develop simple implementations of those methods.

\appendix
\section{Derivatives of~the~discretized functional with respect to~the~physical coordinates}
\label{sect:x:der}

In \cref{sect:formula} we used the computational coordinates $\V{\xi}$ as unknown variables and obtained the physical mesh via linear interpolation.
We can also use the physical coordinates $\V{x}$ as unknown variables and obtain the physical mesh directly (this approach has been considered extensively in the existing literature, e.g.,~\cite{HR11}).
In the following we derive the formulas for the derivatives of the discretized functional with respect to the physical coordinates $\V{x}$ for the convenience of users who prefer this approach.
The formulas are slightly more complicated than those with respect to $\V{\xi}$ and the metric tensor $\M$ has to be updated every time the physical mesh is updated.

To start with, notice that $\Abs{K} I_K$ is now a function of $\V{x}_0^K, \dotsc, \V{x}_d^K$.
We have
\begin{multline}
   \p{(\Abs{K} I_K)}{t}
       = G \p{\Abs{K}}{t}
         + \Abs{K} \tr\left( \p{G}{\J} \p{{(F_K')}^{-1}}{t}\right)
         + \Abs{K} \p{G}{r} \p{{\det(F_K')}^{-1}}{t}
  \\
         + \Abs{K} \sum_{k=1}^d \tr\left( \p{G}{\M} \p{\M}{x^{(k)}} \right)
            \p{x_K^{(k)}}{t} 
         + \Abs{K} \p{G}{\V{x}} \p{\V{x}_K}{t}
   .
   \label{eq:x:der:1}
\end{multline}
In the first term, $\Abs{K} = \frac{1}{d!} \Abs{\det(E_K)} = \frac{1}{d!} \det(E_K) \sgn(\det(E_K))$, where $\sgn$ is the sign function.
From \cref{lem:3.3} we get
\[
   \p{\Abs{K}}{t}
      = \frac{\sgn(\det(E_K))}{d!} \p{\det(E_K)}{t}
      = \frac{\Abs{\det(E_K)}}{d!} \tr\left( E_K^{-1} \p{E_K}{t} \right) 
      = \Abs{K} \tr \left( E_K^{-1} \p{E_K}{t} \right)
   .
\]
For the second term, from \cref{lem:3.1,lem:3.4} and equation \cref{eq:FK:1} we obtain
\begin{align*}
   \tr\left( \p{G}{\J} \p{ {(F_K')}^{-1}}{t} \right)
      & = \tr\left ( \p{G}{\J} \hat{E}_K \p{E_K^{-1}}{t} \right) \\
      & = - \tr\left( \p{G}{\J} \hat{E}_K E_K^{-1} \p{E_K}{t}E_K^{-1} \right) \\
      & = - \tr\left( E_K^{-1} \p{G}{\J} \hat{E}_K E_K^{-1} \p{E_K}{t} \right)
   .
\end{align*}
For the third term, from \cref{lem:3.2} we obtain
\[
   \p{G}{r} \p{ {\det(F_K')}^{-1}}{t}
      = - \p{G}{r} \frac{\det(\hat{E}_K)}{{\det(E_K)}^2} \p{\det(E_K)}{t}
      = - \p{G}{r} \frac{\det(\hat{E}_K)}{\det(E_K)} \tr\left( E_K^{-1} \p{E_K}{t} \right)
   .
\]
Combining the above results  we get
\begin{align}
   \p{(\Abs{K} I_K)}{[\V{x}_1^K,\dotsc,\V{x}_d^K]}
       & = G \Abs{K}  E_K^{-1}
      - \Abs{K} E_K^{-1} \p{G}{\J} \hat{E}_K E_K^{-1}
      - \Abs{K} \p{G}{r} \frac{\det(\hat{E}_K)}{\det(E_K)} E_K^{-1}
   \notag \\
     & \qquad + \Abs{K} {(B_2)}_{d\times d}
      + \Abs{K} {(C_2)}_{d\times d}
   ,
   \label{eq:x:der:2}
\end{align}
and
\begin{align}
   \p{(\Abs{K}I_K)}{\V{x}_0^K} 
      &= - \V{e}^T \left [ G \Abs{K}  E_K^{-1}
      - \Abs{K} E_K^{-1} \p{G}{\J} \hat{E}_K E_K^{-1}
      - \Abs{K} \p{G}{r} \frac{\det(\hat{E}_K)}{\det(E_K)} E_K^{-1} \right ] 
      \notag \\
         & \qquad + \Abs{K} {(B_1)}_{1\times d}
         + \Abs{K} {(C_1)}_{1\times d}
   ,
   \label{eq:x:der:3}
\end{align}
where $B_1$ and $B_2$ are associated with the fourth term in \cref{eq:x:der:1} and $C_1$ and $C_2$ are associated with the fifth term in \cref{eq:x:der:1}.

We first derive $C_1$ and $C_2$. 
Since $\V{x}_K = \frac{1}{d+1} \sum\limits_{k=0}^d \V{x}_{k}^K$, we have
\[
   \p{G}{\V{x}} \p{\V{x}_K}{t} 
      = \frac{1}{d+1} \sum_{k=0}^d \p{G}{\V{x}} \p{\V{x}_{k}^K}{t}
      = \frac{1}{d+1} \sum_{k=0}^d \sum_{l=1}^d 
         \p{G}{x^{(l)}} \p{x_{k}^{K (l)}}{t}
\]
and therefore
\[
   \p{G}{\V{x}} \p{\V{x}_K}{x_{i}^{K (j)}}  
      = \frac{1}{d+1} \sum_{k=0}^d \sum_{l=1}^d \p{G}{x^{(l)}}
         \p{x_{k}^{K (l)}}{x_{i}^{K (j)}}
      = \frac{1}{d+1} \p{G}{x^{(j)}}
   .
\]
This gives
\begin{equation}
   \begin{bmatrix} C_1 \\ C_2 \end{bmatrix}  
      =  \frac{1}{d+1} \begin{bmatrix} \p{G}{\V{x}} \\ \vdots \\ \p{G}{\V{x}} \end{bmatrix}
   .
   \label{eq:x:der:4}
\end{equation}

The main difficulty in computing the fourth term and finding $B_1$ and $B_2$ is that $\M$ is typically defined on a background mesh as a piecewise linear function and therefore its derivatives do not exist on mesh facets, vertices, and edges.
To avoid this difficulty, we assume that $\M$ has been interpolated from the background mesh to the current mesh $\Th$ and the derivative $\p{\M}{x^{(k)}}$ is approximated by that of the interpolating function, i.e., 
\[
   \M = \sum_{j=0}^d \M_{j,K} \phi_{j, K}
   \qquad \text{and} \qquad
   \p{\M}{x^{(k)}} = \sum_{j=0}^d \M_{j,K} \p{\phi_{j,K}}{x^{(k)}}
   ,
\]
where $\M_{j,K}$ is the value of the metric tensor and $\phi_{j,K}$ is the linear basis function at the vertex $\V{x}_{j}^K$.
Then,
\begin{align*}
   \sum_{k=1}^d \tr\left( \p{G}{\M} \p{\M}{x^{(k)}} \right) \p{x_K^{(k)}}{t}
      &= \sum_{k=1}^d \sum_{j=0}^d \tr\left( \p{G}{\M} \M_{j,K} \right)
      \p{\phi_{j,K}}{x^{(k)}} \p{x_K^{(k)}}{t}
   \\
      &=  \sum_{j=0}^d \tr\left( \p{G}{\M} \M_{j,K} \right)  
         \p{\phi_{j,K}}{\V{x}} \p{\V{x}_K}{t}
   ,
\end{align*}
which gives
\begin{equation}
   \begin{bmatrix} B_1 \\ B_2 \end{bmatrix}  
      =  \frac{1}{d+1} \sum_{j=0}^{d} \tr\left( \p{G}{\M} \M_{j, K} \right)
      \begin{bmatrix} \p{\phi_{j,K}}{\V{x}} \\ \vdots \\ \p{\phi_{j,K}}{\V{x}} \end{bmatrix} .
   \label{eq:x:der:5}
\end{equation}

The derivative $\p{\phi_{j,K}}{\V{x}}  = {\left(\nabla \phi_{j, K}\right)}^T$ is computed as follows.
The basis functions satisfy
\[
   \sum_{j=0}^d \phi_{j,K} = 1
   \qquad \text{and} \qquad
   \sum_{j=0}^d \V{x}_{j}^K \phi_{j, K} = \V{x}
   .
\]
Eliminating $\V{x}_{0, K}$ yields
\[
   \sum_{j=1}^d \left( \V{x}_{j}^K-\V{x}_0^K \right) \phi_{j,K} 
      = \V{x} - \V{x}_0^K
\]
and differentiating this with respect to $x^{(k)}$ gives
\[
   \sum_{j=1}^d (\V{x}_{j}^K-\V{x}_0^K)\p{\phi_{j,K}}{x^{(k)}} 
      = \V{e}_k
   ,
\]
where $\V{e}_k$ is the $k^{\text{th}}$ unit vector of $\R^d$.
Hence,
\begin{equation}
 \begin{bmatrix} \p{\phi_{1,K}}{\V{x}} \\ \vdots \\ \p{\phi_{d,K}}{\V{x}}\end{bmatrix} = E_K^{-1}
   \qquad \text{and} \qquad
   \p{\phi_{0,K}}{\V{x}} = - \sum_{j=1}^d \p{\phi_{j,K}}{\V{x}} 
   .
   \label{eq:x:der:6}
\end{equation}

Like (\ref{eq:mmpde:4}), we can write the MMPDE approach of the mesh equation in terms of mesh velocities.
It reads as
\begin{equation}
\label{eq:mmpde:6}
\p{\V{x}_i}{t} = \frac{P_i}{\tau} \sum_{K \in \omega_i} |K| \V{v}_{i_K}^K,\quad i = 1,\dotsc, N_v
\end{equation}
where the balancing parameter $P_i$ is defined in \cref{eqP:1} and the local velocities are given by
\begin{align}
\begin{bmatrix} {(\V{v}_{1}^K)}^T \\ \vdots \\ {(\V{v}_{d}^K)}^T \end{bmatrix}
& = - G E_K^{-1}
      + E_K^{-1} \p{G}{\J} \hat{E}_K E_K^{-1}
      + \p{G}{r} \frac{\det(\hat{E}_K)}{\det(E_K)} E_K^{-1}
\notag \\
& \quad - \frac{1}{d+1} \sum_{j=0}^{d} \tr\left( \p{G}{\M} \M_{j, K} \right)
      \begin{bmatrix} \p{\phi_{j,K}}{\V{x}} \\ \vdots \\ \p{\phi_{j,K}}{\V{x}} \end{bmatrix} 
     - \frac{1}{d+1} \begin{bmatrix} \p{G}{\V{x}} \\ \vdots \\ \p{G}{\V{x}} \end{bmatrix},
\label{eq:mmpde:7}
\\
{(\V{v}_{0}^K)}^T & = - \sum_{k=1}^d {(\V{v}_{k}^K)}^T 
- \sum_{j=0}^{d} \tr\left( \p{G}{\M} \M_{j, K} \right) \p{\phi_{j,K}}{\V{x}}
-  \p{G}{\V{x}} .
\label{eq:mmpde:8}
\end{align}

\printbibliography{}


\end{document}